\documentclass[notitlepage]{amsart}

\newtheorem{pro}{Proposition}[section]
\newtheorem{teo}[pro]{Theorem}
\newtheorem{defi}[pro]{Definition}
\newtheorem{lem}[pro]{Lemma}
\newtheorem{cor}[pro]{Corollary}

\newtheorem{rk}[pro]{Remark}
\newtheorem{ex}[pro]{Example}

\newcommand{\Ext}{\mathrm{Ext}}
\newcommand{\Hom}{\mathrm{Hom}}
\newcommand{\End}{\mathrm{End}}
\newcommand{\A}{\mathcal{A}}

\newcommand{\C}{\mathcal{C}}
\newcommand{\T}{\mathcal{T}}
\newcommand{\X}{\mathcal{X}}
\newcommand{\Y}{\mathcal{Y}}
\newcommand{\Z}{\mathcal{Z}}

\newcommand{\Csf}{\mathsf{C}}
\newcommand{\Ksf}{\mathsf{K}}
\newcommand{\pd}{\mathrm{pd}}
\newcommand{\inj}{\mathrm{inj}}

\newcommand{\ann}{\mathrm{ann}}

\newcommand{\rad}{\mathrm{rad}}
\newcommand{\resdim}{\mathrm{resdim}}
\newcommand{\coresdim}{\mathrm{coresdim}}
\newcommand{\add}{\mathrm{add}}

\newcommand{\gen}{\mathrm{gen}}
\newcommand{\Ker}{\mathrm{Ker}}
\newcommand{\Cone}{\mathrm{Cone}}
\newcommand{\Ima}{\mathrm{Im}}
\newcommand{\findim}{\mathrm{fin.dim}}

\newcommand{\Coker}{\mathrm{CoKer}}
\newcommand{\modu}{\mathrm{mod}}
\newcommand{\homo}{\mathsf{K}^b(\mathrm{proj}(\Lambda))}
\newcommand{\HomK}{\mathrm{Hom}_{\mathsf{K}^b(\mathrm{proj}(\Lambda))}}

\newcommand{\pru}{P^\bullet_{\geq-1}}
\newcommand{\prn}{P^\bullet_{\geq-n}}
\newcommand{\pr}{P^\bullet}
\newcommand{\proj}{\mathrm{proj}}
\newcommand{\cogen}{\mathrm{cogen}}

\usepackage{latexsym,amssymb,amscd,rotating,xypic,amsmath,amsthm,amscd} 

\usepackage[Lenny]{fncychap}

\begin{document}

\title[$n$-term silting complexes]{$n$-term silting complexes in $\homo$}
\author{Luis Mart\'inez and Octavio Mendoza}
\thanks{
The authors thank Project PAPIIT-Universidad Nacional Aut\'onoma de M\'exico IN100520.}

\begin{abstract} Let $\Lambda$ be an Artin algebra and $\homo$ be the triangulated category of bounded co-chain complexes in $\proj(\Lambda).$ It is well known \cite{AIR14} that two-terms silting complexes in $\homo$ are described by the $\tau$-tilting theory.  The aim of this paper is to give a characterization of certain $n$-term silting complexes in $\homo$ which are induced by $\Lambda$-modules. In order to do that, we introduce the notions of $\tau_n$-rigid,  $\tau_n$-tilting and $\tau_{n,m}$-tilting $\Lambda$-modules. The latter is both a generalization of $\tau$-tilting and tilting in $\modu(\Lambda).$  It is also stated and proved some variant, for $\tau_n$-tilting modules, of the well known Bazzoni's characterization for tilting modules \cite[Theorem 3.11]{Ba04}. We give some connections between $n$-terms presilting complexes in $\homo$ and $\tau_n$-rigid $\Lambda$-modules. Moreover, a characterization is given to know when a $\tau_n$-tilting $\Lambda$-module is $n$-tilting. We also study more deeply the properties of  the 
$\tau_{n,m}$-tilting $\Lambda$-modules and their connections of being $m$-tilting in some quotient algebras. We apply the developed $\tau_{n,m}$-tilting theory to the finitistic dimension and thus for $n=m=1,$ we get as a particular case \cite[Theorem 2.5]{Su21}. Finally, at the end of the paper we  discuss and state some open questions (conjectures)  that we consider crucial for the future develop of the $\tau_{n,m}$-tilting theory.
\end{abstract}  
\maketitle

\setcounter{tocdepth}{1}
\tableofcontents

\section{Introduction}

Let $\Lambda$ be an Artin $R$-algebra. We denote by $\modu(\Lambda)$ the category of finitely generated left $\Lambda$-modules and 
by $\proj(\Lambda)$ (respectively, $\inj(\Lambda)$)  the category of finitely generated projective (respectively, injective) left $\Lambda$-modules. 

Following \cite{AR75}, we have the usual duality functor $D_\Lambda:\modu(\Lambda)\to \modu(\Lambda^{op}),$ where 
$D_\Lambda:=\Hom_R(-,k)$ and $k$ is the injective envelope of $R/\rad(R),$ the duality functor 
$*:=\Hom_\Lambda(-,\Lambda):\proj(\Lambda)\to\proj(\Lambda^{op})$ and the Nakayama equivalence 
$\nu:\proj(\Lambda)\to\inj(\Lambda),$ where $\nu(P):=D_{\Lambda^{op}}(P^*).$ For $M\in\modu(\Lambda),$ the class of all the $\Lambda$-modules which are a quotient (respectively, a submodule) of finite direct sums of copies of $M$ is denoted by  $\gen(M)$ (respectively,  
$\cogen(M)).$

Let $\A$ be an additive category. For $A\in\A,$  $\add(A)$ is the class of all the direct summands of finite direct sums of copies of $M.$ We denote by $\mathsf{C}(\A)$ the category of co-chain complexes of objects in $\A.$ For a cochain $X^\bullet\in\mathsf{C}(\A)$ and an integer $m,$ we have that $X^m$ is the $m$-th component and $d^m_{X^\bullet}:X^m\to X^{m+1}$ is the $m$-th differential, respectively, of $X^\bullet$ in $\A.$ If $\A$ is abelian, we have the $m$-cocycle $Z_m(X^\bullet):=\Ker(d^m_{X^\bullet})$ and the $m$-coboundary $B_m(X^\bullet):=\Ima(d^{m-1}_{X^\bullet}),$  respectively, of $X^\bullet$ in $\A.$ Thus, there exists the 
cohomology functor $H^i:\mathsf{C}(\A)\to \A,$ for each integer $i,$ where $H^i(X^\bullet):=Z_i(X^\bullet)/B_i(X^\bullet).$ The homotopy category of cochain complexes of objects in $\A$ is denoted by $\mathsf{K}(\A).$  

For each $M\in\modu(\Lambda),$ we fix a minimal projective resolution of $M$
\begin{center}
$\cdots\to P^{-n}(M)\xrightarrow{\pi^{-n}_M}P^{-n+1}(M)\to \cdots \to P^{-1}(M)\xrightarrow{\pi^{-1}_M} P^0(M)\xrightarrow{\pi^{0}_M} M\to 0.$  
\end{center}
Associated to this minimal projective resolution of $M,$ we have the cochain complex $P^\bullet(M)\in\mathsf{C}(\modu(\Lambda))$ which agrees with this resolution in degree $i\leq 0$ and vanishes for any degree $i>0.$ Note that $H^0(P^\bullet(M))\simeq M$ and $H^i(P^\bullet(M))=0$ for any $i\neq 0.$ The $i$-th syzygy of $M$ is $\Omega^i(M):=\Ker(\pi^{-i+1}_M),$ where $\pi^1_M$ is the zero morphism $M\to 0.$ The Auslander-Reiten translation (AR translation, for short) of $M$ is 
$\tau(M):=\Ker(\nu(\pi^{-1}_M)).$ Following \cite{Iy07}, the $n$-AR translation of $M$ is $\tau_n(M):=\tau(\Omega^{n-1}(M)),$  for any integer $n\geq 1.$ Sometimes we write $\tau_{\small{\Lambda},n}(M)$ to  make it clear that we are computing $\tau_n(M)$ in the category $\modu(\Lambda).$

For any cochain complex $X^\bullet\in\mathsf{C}(\modu(\Lambda))$ and any integer $n,$ we have the truncated cochain subcomplex 
$X^{\bullet}_{\geq n}$ of $X^{\bullet},$ where $X^i_{\geq n}:=X^i$ if $i\geq n$ and $X^i_{\geq n}:=0$ elsewhere. Similarly, we have the truncated cochain complex $X^{\bullet}_{\leq n}$ which is a quotient  of $X^{\bullet}.$

The concept of triangulated category, introduced by J. L. Verdier in \cite{Verdier}, has been very useful for studying the category $\modu(\Lambda)$ for some Artin $R$-algebra $\Lambda.$ This is because certain triangulated categories become quite accessible if one applies the methods of representation theory \cite{Ha88}. The most famous example \cite{Ha88} of a triangulated category is the bounded derived category $\mathsf{D}^b(\modu(\Lambda))$ of cochain complexes over $\modu(\Lambda)$ having bounded cohomology, and can be identified with the  homotopy category $\mathsf{K}^{-,b}(\proj(\Lambda))$ of  cochain complexes which are bounded above and have bounded cohomology over the class $\proj(\Lambda).$ In this paper we will use the homotopy category  $\mathsf{K}^{b,\leq0}(\proj(\Lambda))$ whose objects are all the  
$X^\bullet\in \mathsf{K}^{b}(\proj(\Lambda))$ such that $H^k(X^\bullet)=0$ $\forall\,k>0.$
\

Silting theory is a topic originating from representation theory of Artin algebras which is an extension of tilting theory, it encompasses methods for studying derived equivalences which are widely employed in many areas of research \cite{AH19}. Tilting theory started as a module-theoretic interpretation of the reflection functors introduced by I. N. Bernstein, I. M. Gelfand and V. A. Ponomarev in the early 1970s.
\

Let $\Lambda$ be an Artin algebra. Consider a class $\X\subseteq\modu(\Lambda).$ The right orthogonal class  of $\X$ is $\X^\perp:=\cap_{i>0}\,\X^{\perp_i},$ where   
$\X^{\perp_i}:=\{M\in\modu(\Lambda)\;|\;\Ext^i_\Lambda(-,M)|_{\X}=0\}$ is the $i$-th right orthogonal class of $\X,$ for each positive integer $i.$ Dually, we have the left orthogonal class ${}^{\perp}\X$ and the $i$-th left orthogonal class ${}^{\perp_i}\X$ of $\X.$ Following Y. Miyashita in \cite{Mi86}, it is said that $T\in\modu(\Lambda)$ is {\bf $n$-tilting} if  $\pd(T)\leq n,$ $T\in T^\perp$ and there is an exact sequence $0\to {}_\Lambda\Lambda\to T_0\to T_1\to\cdots\to T_n\to 0$ with $T_i\in\add(T)$ for all $i\in[0,n].$ We just say that $T\in\modu(\Lambda)$ is {\bf tilting} if it is $n$-tilting for some $n.$
\

Consider $M\in\modu(\Lambda).$ We denote by $rk(M)$ the number of pairwise non-isomorphic indecomposable direct summands of $M.$ Following \cite{AIR14}, it is said that $M$ is {\bf $\tau$-rigid} if $\Hom(M,\tau(M))=0.$ Moreover, $M$ is called {\bf $\tau$-tilting} if it is $\tau$-rigid and $rk(M)=rk({}_\Lambda\Lambda).$ 
\

Let $\Z$ be a class of objects in the homotopy category $\homo$ of bounded cochain complexes in $\proj(\Lambda).$ For each positive integer $i,$ the  right $i$-th orthogonal class of $\Z$ is  $\Z^{\perp_i}:=\{P^\bullet\in\homo\;|\;\HomK(-,P^\bullet[i])|_{\Z}=0\};$ we also consider the right orthogonal class $\Z^{\perp_{>i}}:=\cap_{j>i}\Z^{\perp_j}.$  Dually, we have the left orthogonal classes ${}^{\perp_{>i}}\Z$  and ${}^{\perp_i}\Z$  of $\Z.$ Following \cite{AI12, AH19}, we recall that a cochain complex $P^\bullet\in \homo$ is {\bf presilting} if $P^\bullet\in(P^\bullet)^{\perp_{>0}}.$ Moreover, $P^\bullet$ is {\bf silting} if it is presilting and $\homo$ is the smallest triangulated category containing 
$\add(P^\bullet).$ By  \cite[Propositions 3.6(a) and 3.7 (a)]{AIR14}, we get that $M$ is a $\tau$-tilting $\Lambda$-module if, and only if, the truncated complex $P^\bullet_{\geq -1}(M)$ is silting. Thus, it is quite natural to characterize when $\prn(M)$ is presilting (silting) for an arbitrary $M\in\modu(\Lambda)$ and a positive integer $n.$

In order to state the main definitions, for any $M\in\modu(\Lambda),$ we consider the quotient Artin algebra $\Gamma_{\!\!M}:=\Lambda/\ann(M),$ where $\ann(M)$ is the annihilator of $M$ in $\Lambda.$ In particular $M\in \modu(\Gamma_{\!\!M}).$

\begin{defi} Let $\Lambda$ be an Artin Algebra,  $M\in\modu(\Lambda),$ 
$\Gamma_{\!\!M}:=\Lambda/\ann(M)$ and $m,n\geq 1.$
\begin{itemize} \item[(a)] The {\bf $\tau_n$-perpendicular category} $\mathcal{X}^{\perp_{\tau_n}}$ of $\mathcal{X}\subseteq\modu(\Lambda)$ is $$\mathcal{X}^{\perp_{\tau_n}}:={}^{\perp_0}\tau_n(\mathcal{X})\cap\cap_{i=1}^{n-1}\mathcal{X}^{\perp_i}.$$ 
\item[(b)]  $M$ is {\bf $\tau_n$-rigid} in $\modu(\Lambda)$  if $M\in M^{\perp_{\tau_n}}$. 
\item[(c)]  $M$ is {\bf $\tau_n$-tilting}  in $\modu(\Lambda)$ if $\prn(M)$ is silting in $\homo$.
\item[(d)] $M$ is  $\tau_{n,m}$-tilting in $\modu(\Lambda)$ if $M$ is $\tau_n$-tilting in $\modu(\Lambda)$ and $\tau_m$-rigid  in 
$\modu(\Gamma_{\!\!M}).$ 
\end{itemize}
\end{defi}

We recall that the term ``$\tau_n$-rigid" was firstly introduced by Karin M. Jacobsen and Peter Jorgesen in \cite{JJ20}. However, $\tau_n$-rigid for $M\in\modu(\Lambda)$ in the sense of \cite{JJ20}  means that $\Hom_\Lambda(M,\tau_nM)=0.$

Let $M\in\modu(\Lambda).$ Note that $M$ is $\tau_1$-rigid if, and only if, it is $\tau$-rigid. By Corollary \ref{CoroNAIR, 3.4}, we know that
${}^{\perp_0}\tau_n(M)\subseteq M^{\perp_n},$ and moreover, ${}^{\perp_0}\tau_n(M)= M^{\perp_n}$ if $\pd(M)\leq n.$ 
 We point out, from \cite[Remark 3.3, Corollary 3.7]{We13} that $M$ is tilting in 
 $\modu(\Lambda)$ if, and only if, $P^\bullet(M)$ 
is silting in $\homo.$ In particular, any tilting $\Lambda$-module $T$ is  $\tau_{n,n}$-tilting, for $n:=\max\{1,\pd(T)\},$ since $T$ 
is faithful. On the other hand, by Proposition \ref{n=1}, we have that  
$\tau_{1,1}$-tilting coincides with the notion of $\tau$-tilting.  Note that there exist $\tau_{2,2}$-tilting $\Lambda$-modules of infinity projective dimension which are not $\tau$-rigid, see Example \ref{ExImp}. 
\

This paper has several main results. The first of them (see Theorem \ref{NAIR, 3.4}) states that $M$ is $\tau_n$-rigid in $\modu(\Lambda)$ if, and only if, the $n$-term complex $\prn(M)$ is presilting in $\homo.$ With respect to ``Bazzoni's characterization'' of a $\tau_n$-tilting $M,$ 
we proved that $\gen_n(M)=M^{\perp_{\tau_n}}$ (see Theorem \ref{p4. NAIR, 2.2}) which is  relevant for the development of $\tau_{n,m}$-tilting theory. In Theorem \ref{c1}, we proved that $M$ is sincere and $rk(M)=rk(\Lambda)$ if $M$ is $\tau_n$-tilting in $\modu(\Lambda).$ Moreover, for a $\tau_n$-tilting, we give in Theorem \ref{TEO} a list of conditions which are equivalent to the property of being $n$-tilting. It is worth mentioning that if we take $n=1$ in Theorem \ref{TEO}, we get as a particular case \cite[Theorem 1.3 and Theorem 3.2]{Zh22}. In Section 6, we study more deeply the $\tau_{n,m}$-tilting modules. For example (see Theorem \ref{NAIR, 2.2}) If $M$ is $\tau_{n,m}$-tilting, then $M$ is $m$-tilting in the quotient algebra $\Gamma_{\!\!M}.$  However, if $M$ is $\tau_n$-rigid and sincere in $\modu(\Lambda),$ we were able to prove (only for $n=1,2$) that $M$ is $\tau_n$-tilting. The general case remains open (as can be seen in the Conjecture 3, at the end of the paper) and it is very closely connected with the so called ``$\tau_{n,m}$-tilting correspondence theorem'' (see Theorem \ref{corresT}). One application that can be obtained from 
Theorem \ref{NAIR, 2.2} is related with the finitistic dimension, see Corollary \ref{Corfindim}, which says that $\findim(\Lambda)\leq \findim(\End_{\Gamma_{\!\!M}}(M)^{op})+\pd_\Lambda(M)+m$ if $M$ is $\tau_{n,m}$-tilting in $\modu(\Lambda).$ It is also introduced the notion of $\tau_{n,m}$-compatible class in $\modu(\Lambda),$ see Definition \ref{compClass}. In Theorem \ref{Comp-tau}, we proved that if some class $\omega\subseteq\modu(\Lambda)$ is $\tau_{n,m}$-compatible, then there exists a basic sincere and $\tau_n$-rigid $T\in\modu(\Lambda)$  which is $\min\{m,n\}$-tilting in $\modu(\Gamma_{\!T})$ and $\omega=\add(T).$ Finally, Theorem \ref{EquiTaurig} gives us a very closely connection between $\tau_{n,m}$-tilting $\Lambda$-modules , $\tau_{n,m}$-compatible classes in $\modu(\Lambda)$ and $m$-tilting in a quotient algebra.

\section{Background and preliminary results}

We start this section by collecting all the background material that will be necessary in the sequel. First, we introduce some general notation.

Let $\Lambda$ be an Artin algebra and $\X\subseteq\modu(\Lambda).$ Following \cite{BMPS19}, we  say that the class $\X$ is {\bf resolving} if  it is closed under extensions,  epi-kernels (that is, under taking kernels of epimorphisms between objects in $\X$) and $\proj(\Lambda)\subseteq \X.$   If in addition $\X$ is closed under direct summands, we say that $\X$ is {\bf left saturated}. If the dual properties hold true for $\X,$ we say that $\X$ is {\bf coresolving} ({\bf right saturated}). The class $\X$ is {\bf left (right) thick}  if it is closed under extensions, epi-kernels (mono-cokernels) and direct summands.

The following is an useful characterization of $\tau$-rigid $\Lambda$-modules in terms of the bifunctor $\Ext^1_\Lambda(-,-).$

\begin{pro}\cite[Proposition 5.8]{AS81}\label{AS81, 5.8} For an Artin algebra $\Lambda$ and $M,N\in\modu(\Lambda),$ we have that  
$\Hom_\Lambda(M,\tau(N))=0$ $\Leftrightarrow$ $N\in{}^{\perp_1}\gen(M).$\end{pro}

Following \cite{AB89}, for $M\in\modu(\Lambda)$ and $\X\subseteq\modu(M),$ an {\bf $\X$-coresolution} (of length $n$) of $M$ is an exact sequence $0\to M\to X_0\to X_1\to\cdots\to X_n\to 0$ with $X_i\in\X$ for all $i\in[0,n].$ The class of all the $M\in\modu(\Lambda)$ admitting an $\X$-coresolution of length $n$ is denoted by $\X^{\vee}_n.$ We also consider the class 
$\X^{\vee}:=\cup_{n\in\mathbb{N}}\;\X^{\vee}_n.$ The $\X$-{\bf{coresolution dimension}} 
of $M$ is $\coresdim_\X(M):=\min\{n\in\mathbb{N}\;:\; M\in \X^{\vee}_n\}.$ In case  $M\not\in \X^{\vee},$ we set 
$\coresdim_\X(M):=\infty.$ Dually, we have $\X$-resolutions of length $n,$ the classes $\X^\wedge$ and $\X^{\wedge}_n,$ and the 
$\X$-resolution dimension $\resdim_\X(M)$ of $M.$
\

Let $\C$ be an additive category and $\X$ be a class of objects in $\C.$ Following \cite{AS80}, it is said that a morphism $f:X\rightarrow C$ is  an $\X$-{\bf{precover}} if $X\in\X$ and $\mathrm{Hom}_\C(Z,f):\mathrm{Hom}_\C(Z,X)\rightarrow\mathrm{Hom}_\C(Z,C)$ is surjective for any $Z\in\X$. The class $\X$ is {\bf{precovering}} if every $C\in\C$ admits a $\X$-precover. Dually, we have the notion of $\X$-{\bf{preenvelope}} and {\bf{preenveloping}} class .
\

\begin{teo}\cite{AR91}\label{AR91, 5.5}  For an Artin algebra $\Lambda,$ the following statements hold true.
\begin{itemize}
\item[(a)] The map $T\mapsto T^\perp$ induces a bijection between iso-classes of basic tilting $\Lambda$-modules and preenveloping right saturated classes $\mathcal{X}\subseteq\modu(\Lambda)$ with $\modu(\Lambda)=\X^\vee.$
\item[(b)] If $T\in\modu(\Lambda)$ is tilting, then $\add(T)$ is a relative generator in $T^\perp.$
\end{itemize}
\end{teo}
\begin{proof} The item (a) follows from the dual of \cite[Theorem 5.5 (a)]{AR91},  and the item (b) from the dual of \cite[Theorem 5.4 (b)]{AR91}
\end{proof}

We point out the following relationship between $\tau$-rigid modules and presilting objects in $\homo.$

\begin{lem}\cite[Lemma 3.4]{AIR14}\label{AIR14, 3.4} For $M$, $N\in\modu(\Lambda),$ the following conditions are equivalent.
\begin{itemize}\item[(a)] $\Hom_\Lambda(N,\tau(M))=0$.
\item[(b)] $\HomK(\pru(M),\pru(N)[1])=0$.\end{itemize}
In particular, $M$ is $\tau$-rigid if, and only if, $\pru(M)$ is presilting in $\homo.$\end{lem}

The following result shows the relationship between silting subcategories in $\homo$ and n-tilting $\Lambda$-modules.

\begin{pro}\cite[Remark 3.3, Theorem 3.5, Corollary 3.7]{We13}\label{We13, 3.7} Let $\Lambda$ be an Artin algebra and $M\in\modu(\Lambda).$ Then $M$ is tilting in $\modu(\Lambda)$ if, and only if, $\pr(M)$ is silting in $\homo.$
\end{pro}

For $M\in\modu(\Lambda)$ and $n\geq1,$ we denote by $\gen_n(M)$ the class of all the $X\in\modu(\Lambda)$ admitting an exact sequence $M_n\to M_{n-1}\to \cdots\to M_2\to M_1\to X\to 0,$ where $M_i\in\add(M)$ $\forall\,i\in[1,n].$

The  Bazzoni's characterization for $n$-tilting modules (in the sense of Angeleri-Coelho) is given in \cite[Theorem 3.11]{Ba04}.  In what follows we give the version for $n$-tilting modules (in the sense of Miyashita) in Artin algebras that can be found in \cite[Proposition 5.10, Theorem 3.14]{AM21}.

\begin{teo}  \label{Ba04, 3.11}  Let $\Lambda$ be an Artin algebra,  $M\in\modu(\Lambda)$ and $n\geq 1.$ Then: $M$ is $n$-tilting $\Leftrightarrow$  $\gen_n(M)=M^\perp$ $\Leftrightarrow$ $\gen_k(M)=M^\perp$ $\forall\,k\geq n.$
\end{teo}

 We exhibit now the relation between exact sequences in $\modu(\Lambda)$ and distinguished triangles in $\mathsf{K}(\proj(\Lambda)).$

\begin{pro}\cite[Proposition 3.5.14, Lemma 3.5.50, Proposition 3.5.51]{Zi14}\label{Zi} Let $0\to N\overset{f}{\to}M\overset{g}{\to}K\to0$ be an exact sequence in $\modu(\Lambda).$ Then there is a distinguished triangle $\pr(N)\xrightarrow{f^\bullet}\pr(M)\xrightarrow{g^\bullet}\pr(K)\to\pr(N)[1]$ in $\mathsf{K}(\proj(\Lambda)),$ such that $H^0(f^\bullet)=f$ and $H^0(g^\bullet)=g.$
\end{pro}

To finish this section, we recall from \cite[Section 3]{MSSS13} the notion of resolution (coresolution) dimension in a triangulated category $\T.$ For $\X$, $\Y\subseteq\T,$ we denote by $\X*\Y$ the class of all the objects $Z\in\T$ admitting  a distinguished triangle $X\to Z\to Y\to X[1]$ in 
$\T,$ with $X\in\X$ and $Y\in\Y.$
\

For each $n\in\mathbb{N}$ and $\X\subseteq\T,$ the class $\epsilon^\wedge_n(\X)\subseteq\T$ is defined recursively as follows:
 $\epsilon^\wedge_0(\X):=\X$ and  $\epsilon^\wedge_{n}(\X):=\X*\epsilon^\wedge_{n-1}(\X)[1]$ if $n\geq1.$ The $\X$-resolution dimension of 
 $M\in\T$ is $\resdim_\X(M):=\min\{n\in\mathbb{N}\;:\; M\in \epsilon^\wedge_{n}(\X)\}.$ Dually, we have the class $\epsilon^\vee_n(\X)$  and the $\X$-coresolution dimension $\coresdim_\X(M)$ of 
 $M\in\T.$
 \
 
Let  $\T:=\homo$ and $P^\bullet\in\homo.$ We set $a(P^\bullet):=\inf\{n\in\mathbb{Z}\;:\; P^n\neq 0\}$ and $b(P^\bullet):=\sup\{n\in\mathbb{Z}\;:\; P^n\neq 0\}.$ If 
$P^\bullet\neq 0,$ the integer $\ell(P^\bullet):=b-a+1$ is known as the length of the complex  $P^\bullet.$ In the trivial case $P^\bullet=0,$ we set $\ell(P^\bullet):=0.$ For the class $\X:=\add(\Lambda[0]),$ we will need later to find a bound of the resolution dimension $\resdim_\X(P^\bullet)$ in $\T$ for some complex  $P^\bullet\in \Ksf^{b}(\proj(\Lambda)),$ with $b(P^\bullet)\leq 0.$ Therefore, the following lemma will be very useful.

\begin{lem}\label{lNAIR, 2.2} Let $m\in\mathbb{Z}$ and $\X:=\add(\Lambda[-m])$ in $\homo.$ Then
\begin{center}
$\resdim_\X(P^\bullet)\leq m-\min\{m,a(P^\bullet)\},$ for all $P^\bullet\in\Ksf^b(\proj(\Lambda)$ with $b(P^\bullet)\leq m.$
\end{center}
\end{lem}
\begin{proof}  Let $P^\bullet\in\Ksf^b(\proj(\Lambda)$ with $b(P^\bullet)\leq m.$ 
We will proceed by induction on $n:=m-\min\{m,a(P^\bullet)\}.$ 
For $n=0,$ it is clear that $a(P^\bullet)=m=b(P^\bullet)$ and thus 
$P^\bullet\in\X.$ In particular, $\resdim_\X(P^\bullet)=0.$
\

Let $n\geq 1.$ Thus $P^\bullet\neq 0$ and $a(\pr)<m.$ Consider the exact sequence $0\to P^m[-m]\to\pr\to\pr_{\leq m-1}\to0$ in $\mathsf{C}(\proj(\Lambda)).$ Then, we get the distinguished triangle 
$\eta:\;\pr_{\leq m-1}[-1]\to P^m[-m]\to\pr\to\pr_{\leq m-1}$ in $\homo.$
Note that $a(\pr_{\leq m-1}[-1])=a(\pr)+1$ and $b(\pr_{\leq m-1}[-1])\leq m.$ Hence, by the inductive hypothesis $\resdim_\X(\pr_{\leq m-1}[-1])\leq m-a(\pr)-1.$ Therefore, from $\eta,$ we conclude that $\resdim_\X(P^\bullet)\leq m-a(\pr)$ since 
$P^m[-m]\in\X.$
\end{proof}

\section{$\tau_n$-rigid modules}

The main objective of this section is to study some properties of the $\tau_n$-rigid $\Lambda$-modules  and  establish a connection with the presilting objects in $\homo.$ The main idea is to give a generalization of the  relationship between $\tau$-rigid $\Lambda$-modules and the 2-terms presilting complexes in $\homo$ \cite{AIR14}.

\begin{lem}\label{l1, NAIR, 3.4}For an Artin algebra $\Lambda,$  $m\in\mathbb{Z}$, $n\in\mathbb{N}^+,$ $P^\bullet\in\mathsf{C}(\proj(\Lambda))$ and $Q^\bullet\in\mathsf{C}(\modu(\Lambda)),$ the following statements hold true.
\begin{itemize}
\item[(a)] Let $f^\bullet: P^\bullet\to Q^\bullet$ be a morphism in $\mathsf{C}(\modu(\Lambda)),$ where $P^\bullet=P^\bullet_{\geq m},$  $Q^\bullet=Q^\bullet_{\leq m+n-1}$ and $H^{m+h}(Q^\bullet)=0$ $\forall\, h\in [0,n-1).$ If there exist morphisms of $\Lambda$-modules $s^{m+n}: P^{m+n}\to Q^{m+n-1}$  and $s^{m+n-1}: P^{m+n-1}\to Q^{m+n-2}$ such that  $f^{m+n-1}=s^{m+n}d^{m+n-1}_{P^\bullet}+d^{m+n-2}_{Q^\bullet}s^{m+n-1},$ then $f^\bullet=0$ in $\mathsf{K}(\modu(\Lambda)).$
\item[(b)] Let  $g^\bullet:P^\bullet_{\geq m}\to Q^\bullet_{\geq m}$ be a morphism in $\mathsf{C}(\modu(\Lambda)).$ 
If  $H^i(Q^\bullet)=0$ $\forall\,i\in[m-n+1,m],$ then there exists a morphism of cochain complexes $\overline{g}^\bullet:P^\bullet_{\geq m-n}\to Q^\bullet_{\geq m-n}$ such that $\overline{g}^\bullet|_{P^\bullet_{\geq m}}=g^\bullet.$
\end{itemize}
\end{lem}
\begin{proof} Consider the factorization $Q^{m-1}\xrightarrow{p}B_m(Q^\bullet)\xrightarrow{i}Q^m$ of $d^{m-1}_{Q^\bullet}$ trough its image.
\

(a) Note that $f^\bullet: P^\bullet\to Q^\bullet$ can be seen as follows
$$\small{ {\xymatrix{0\ar[r]\ar[d]&P^m\ar[d]^{f^m}\ar[r]&\cdots\ar[r]&P^{m+k}\ar[r]^{d^{m+k}_{P^\bullet}}\ar[d]^{f^{m+k}}&P^{m+k+1}\ar[r]\ar[d]^{f^{m+k+1}}&\cdots\ar[r]&P^{m+n}\ar[d]\\
Q^{m-1}\ar[r]_{d^{m-1}_{Q^\bullet}}&Q^m\ar[r]&\cdots\ar[r]&Q^{m+k}\ar[r]_{d^{m+k}_{Q^\bullet}}&Q^{m+k+1}\ar[r]&\cdots\ar[r]&0.}}}$$
We use induction on $n$ to prove that $f^\bullet=0$ in $\mathsf{K}(\modu(\Lambda)).$ For $n=1,$ there is nothing to proof. 
\

Let $n\geq 2.$ Consider the complex $P'^\bullet:=P^\bullet_{\geq m+1}$ and the morphism $f'^\bullet:P'^\bullet\to Q^\bullet$ of complexes which is the restriction of $f^\bullet$ on $P'^\bullet.$ By the inductive hypothesis, we have that $f'^\bullet=0$ in  $\mathsf{K}(\modu(\Lambda)).$ Then, there is a family $\{s^{m+k}:P^{m+k}\to Q^{m+k-1}\}_{k=1}^n$ of $\Lambda$-morphisms such that $f^{m+k}=s^{m+k+1}d^{m+k}_{P^\bullet}+d^{m+k-1}_{Q^\bullet}s^{m+k}$ for all $k\in[1,n].$ Since $f^{m+1}=s^{m+2}d^{m+1}_{P^\bullet}+d^{m}_{Q^\bullet}s^{m+1},$ we get $d^m_{Q^\bullet}(f^m-s^{m+1}d^m_{P^\bullet})=0$ and hence there is $r:P^m\to B_m(Q^\bullet)$ such that $ir=f^m-s^{m+1}d^m_{P^\bullet}.$ Using now that $P^m$ is projective, we have $s^m:P^m\to Q^{m-1}$ with $r=ps^m.$ Therefore $f^m=s^{m+1}d^m_{P^\bullet}+d^{m-1}_{Q^\bullet}s^m$ and thus $f^\bullet=0$ in $\mathsf{K}(\modu(\Lambda))$ since $P^\bullet=P^\bullet_{\geq m}.$
\

(b) Since $d^m_{Q^\bullet}g^m d^{m-1}_{P^\bullet}=d^m_{Q^\bullet}d^{m-1}_{Q^\bullet}g^{m-1}=0,$ we get $t:P^{m-1}\to B_m(Q^\bullet)$ such that $it=g^m d^{m-1}_{P^\bullet}.$ Using now that $P^{m-1}$ is projective, there is some $g^{m-1}: P^{m-1}\to Q^{m-1}$ with $p g^{m-1}=t$ and thus $d^{m-1}_{Q^\bullet}g^{m-1}=g^m d^{m-1}_{P^\bullet}.$ By repeating the above procedure, we obtain morphisms $g^{m-i}:P^{m-i}\to Q^{m-i}$ for each $i\in[1,n],$ and hence a morphism of cochain complexes $\overline{g}^\bullet:P^\bullet_{\geq m-n}\to Q^\bullet_{\geq m-n}$ such that $\overline{g}^\bullet|_{P^\bullet_{\geq m}}=g^\bullet.$
\end{proof}

The following lemma will be very useful in different parts of the paper.

\begin{lem}\label{lemtrunc} Let $n,j\in\mathbb{N}^+$ and $P^\bullet,Q^\bullet\in\Ksf^{\leq 0}(\proj(\Lambda)).$ Then 
$$\Hom_{\Ksf(\proj(\Lambda))}(P^\bullet_{\geq -n},Q^\bullet_{\geq -n}[j])\simeq \Hom_{\Ksf(\proj(\Lambda))}(P^\bullet_{\geq -n},Q^\bullet[j]).$$
\end{lem}
\begin{proof} The exact sequence $0\to Q^\bullet_{\geq -n}[j]\to Q^\bullet[j]\to Q^\bullet_{\leq -n-1}[j]\to 0$ in $\Csf(\proj(\Lambda))$ induces the distinguished triangle $\Ksf(\proj(\Lambda))$
\begin{center}
$Q^\bullet_{\leq -n-1}[j-1]\to Q^\bullet_{\geq -n}[j]\to Q^\bullet[j]\to Q^\bullet_{\leq -n-1}[j].$
\end{center}
  Then, by applying  $(P^\bullet_{\geq -n},-):=\Hom_{\Ksf(\proj(\Lambda))}(P^\bullet_{\geq -n},-)$ to this triangle and since $(P^\bullet_{\geq -n},Q^\bullet_{\leq -n-1}[j-1])=0$ and $(P^\bullet_{\geq -n},Q^\bullet_{\leq -n-1}[j])=0,$ 
 we get the result.
\end{proof}

\begin{pro}\label{p1, NAIR, 3.4} For an Artin algebra $\Lambda$ and $M$, $N\in\modu(\Lambda)$, the following statements hold true.
\begin{itemize}
\item[(a)] Let $1\leq j \leq n.$ Then $\Hom_\Lambda(N,\tau_j(M))=0$ implies that $$\HomK(\prn(M),\prn(N)[j])=0.$$ Moreover, the reverse implication holds true if $j=n.$
\item[(b)] Let $n\geq2$ and $1\leq j<n.$ Then $\HomK(\prn(M),\prn(N)[j])=0$ if, and only if, $\Ext^j_\Lambda(M,N)=0.$
\end{itemize}
\begin{proof}(a) Let $1\leq j<n$ and $f^\bullet: P^\bullet_{\geq-n}(M)\to P^\bullet_{\geq-n}(N)[j]$ in $\mathsf{C}(\proj(\Lambda)).$ This morphism induces the following diagram 
$$\xymatrix{&P^{-j}(M)\ar[d]_{f^{-j}}\ar[r]^{\pi^{-j}_M}&P^{-j+1}(M)\\
P^{-1}(N)\ar[r]_{\pi^{-1}_N}&P^0(N).}$$
Since $\Hom_\Lambda(N,\tau_j(M))=0$ and $\tau_j(M)=\tau(\Omega^{j-1}(M)),$ we get from Lemma \ref{AIR14, 3.4} that $\Hom_{\Ksf^b(\proj(\Lambda))}((P^{-j}(M)\to P^{-j+1}(M)),(P^{-1}(N)\to P^0(N))[1])=0.$ Thus, there exist morphisms $s: P^{-j+1}(M)\to P^0(N)$ and  $s': P^{-j}(M)\to P^{-1}(N)$ such that $f^{-j}=s\pi^{-j}+\pi^{-1}s'.$ Then, by Lemma \ref{l1, NAIR, 3.4} (a), we get that $f^\bullet=0$ in $\homo.$
\

For the case $j=n,$ we have that $\Hom_{\Ksf^b(\proj(\Lambda))}(\prn(M),\prn(N)[j])$ is equal to $\Hom_{\Ksf^b(\proj(\Lambda))}((P^{-n}(M)\to P^{-n+1}(M)),((P^{-1}(N)\to P^{0}(N))[1]) .$ Therefore, by Lemma \ref{AIR14, 3.4}, we get that $\Hom_{\Ksf^b(\proj(\Lambda))}(\prn(M),\prn(N)[n])=0$ if, and only if, $\Hom_\Lambda(N,\tau_n(M))=0.$
\

(b)  Let $n\geq2$ and $1\leq j<n$.  We recall the well known isomorphism 
$$(*)\quad \Ext^j_\Lambda(M,N)\simeq \Hom_{\Ksf(\proj(\Lambda))}(P^\bullet(M),P^\bullet(N)[j]).$$

$(\Rightarrow)$ Assume that $\HomK(\prn(M),\prn(N)[j])=0.$ Let us prove that $\Ext^j_\Lambda(M,N)=0.$ By $(*)$ it 
is enough to prove that $(P^\bullet(M),P^\bullet(N)[j])=0.$ Indeed, let $f^\bullet: \pr(M)\to\pr(N)[j]$ be in $\Csf(\proj(\Lambda)).$ Consider the restriction $g^\bullet:=f^\bullet|_{P^\bullet_{\geq -n}(M)}:P^\bullet_{\geq -n}(M)\to \pr(N)[j].$ Since, by Lemma \ref{lemtrunc}, 
$g^\bullet=0$ in $\Ksf(\proj(\Lambda))$ and 
$H^k(\pr(N)[j])=0$ $\forall\,k\leq -j-1,$ we get by Lemma \ref{l1, NAIR, 3.4} (a) that $f^\bullet=0$ in $\Ksf(\proj(\Lambda)).$
\

$(\Leftarrow)$ Assume that $ (P^\bullet(M),P^\bullet(N)[j])=0.$ Let $f^\bullet: P^\bullet_{\geq -n}(M)\to P^\bullet(N)[j]$ be in $\Csf(\proj(\Lambda)).$ Since $H^k(\pr(N)[j])=0$ $\forall\,k\leq -j-1,$ we get by Lemma \ref{l1, NAIR, 3.4} (b) the existence of $\overline{f}^\bullet: P^\bullet(M)\to P^\bullet(N)[j]$ in $\Csf(\proj(\Lambda))$ such that $\overline{f}^\bullet|_{P^\bullet_{\geq -n}(M)}=f^\bullet.$ Using that $\overline{f}^\bullet=0$ in $\Ksf(\proj(\Lambda)),$ we get that $f^\bullet=0$ in $\Ksf(\proj(\Lambda)).$\end{proof}\end{pro}

\begin{teo}\label{NAIR, 3.4}Let $\Lambda$  be an Artin algebra, $n\geq 1$ and  $M,N\in\modu(\Lambda).$ Then $\prn(N)\in\prn(M)^{\perp_{>0}}$ in $\Ksf^b(\proj(\Lambda))$ if, and only if, $N\in M^{\perp_{\tau_n}}.$ In particular, $\prn(M)$ is presilting in $\homo$ if, and only if, $M\in M^{\perp_{\tau_n}}$.
\begin{proof} Note firstly that $\Hom_{\Ksf^b(\proj(\Lambda))}(\prn(M),\prn(N)[j])=0$ for $j>n.$ Then, by Proposition \ref{p1, NAIR, 3.4}, 
we get that $\prn(N)\in\prn(M)^{\perp_{>0}}$ in $\Ksf^b(\proj(\Lambda))$ if, and only if, 
$N\in {}^{\perp_0}\tau_n(M)\cap\cap_{i=1}^{n-1}\,M^{\perp_i}=M^{\perp_{\tau_n}}.$
\end{proof}\end{teo}

The following result shows in particular that the  $\tau_n$-rigid $\Lambda$-modules are a generalization of the partial $n$-tilting $\Lambda$-modules.

\begin{cor}\label{CoroNAIR, 3.4}  For  $M\in\modu(\Lambda)$ and $j\geq 1,$ the following statements hold true.
\begin{itemize}
\item[(a)] ${}^{\perp_0}(\tau_j(M))\subseteq M^{\perp_j}$ and thus $M^{\perp_{\tau_j}}\subseteq\cap_{i=1}^jM^{\perp_i}.$
\item[(b)] If $\pd(M)\leq j$ then $^{\perp_0}\tau_j(M)=M^{\perp_j}$ and $M^{\perp_{\tau_j}}=M^\perp.$
\item[(c)] $X\in {}^{\perp_j}\,\gen(M)\; \Leftrightarrow\; M\in {}^{\perp_0}\tau_j(X).$
\item[(d)] $M^{\perp_{\tau_j}}$ right thick.
\end{itemize}
\end{cor}
\begin{proof} (a) It follows from Proposition \ref{p1, NAIR, 3.4} by considering $n:=j+1.$ 
\

(b) Note that $\pd(\Omega^{n-1}(M))\leq 1$ since $\pd(M)\leq n.$ Then, for any $N\in\modu(\Lambda),$ by  
\cite[Corollary IV.2.14]{ASS06} we have 
$$\Ext^n_\Lambda(M,N)\simeq \Ext^1_\Lambda(\Omega^{n-1}(M),N)\simeq D\Hom_\Lambda(N,\tau_n(M))$$
and thus $^{\perp_0}\tau_n(M)=M^{\perp_n}.$ On the other hand, $\pd(M)\leq n$ implies that $M^\perp=\cap_{i=1}^n\,M^{\perp_i}.$ Hence $M^{\perp_{\tau_n}}=M^\perp$ since $^{\perp_0}\tau_n(M)=M^{\perp_n}.$
\

(c) It follows from Proposition \ref{AS81, 5.8} using that $\tau_j(X)=\tau(\Omega^{j-1}M).$
\

(d) It is clear that $M^{\perp_{\tau_j}}$ is closed under extensions and direct summands. In order to show that it is also closed under mono-cokernels, we use the inclusion $M^{\perp_{\tau_j}}\subseteq\cap_{i=1}^jM^{\perp_i}$ from (a).
\end{proof}

The next example shows a $\tau_n$-rigid $\Lambda$-module of infinite projective dimension.

\begin{ex} Let $\Lambda$ be an Artin algebra and $S$ be a simple-injective $\Lambda$-module of infinite projective dimension. Then $S$ is $\tau_n$-rigid for any $n\in\mathbb{N}^+.$ Indeed, since $S$ is injective, it is enough to show that $\Hom_\Lambda(S,\tau_n(S))=0.$ Suppose there is a non-zero morphism $f:S\to\tau_n(S).$ Since $S$ is simple-injective and $f\neq 0,$ we have that the injective $S$ is a direct summand of $\tau_n(S)$ contradicting \cite[ IV.1 Proposition 1.10 (c)]{ARS95}.
\end{ex}

In \cite[Theorem 2.10]{AIR14}, it was shown that any $\tau$-rigid $\Lambda$-module  $M$ admits a $\tau$-rigid $\Lambda$-module $N$ such that $M\oplus N$ is $\tau$-rigid and $rk(M)+rk(N)=rk(\Lambda).$ The following example shows that this is not necessarily true for $\tau_n$-rigid $\Lambda$-modules, with $n\geq 2.$

\begin{ex}\label{Ejp1}\cite[Example 2]{RS89} Let $k$ be a field and $\Lambda$ be the quotient path $k$-algebra $kQ/I,$ where $Q$ is the following quiver and $I:=\langle\beta\alpha,\alpha\delta,\delta\gamma\rangle$ 
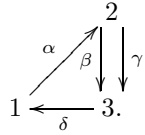
\begin{center}
$\xymatrix{& 2\ar@<1ex>[d]^\gamma \ar@<-1ex>[d]_\beta\\
1\ar[ru]^\alpha & 3.\ar[l]^\delta}$
$\qquad$ 
\end{center}
Note that the simple $S(1)$ has projective dimension 2 and $S(1)\in S(1)^\perp.$ Thus by Corollary \ref{CoroNAIR, 3.4} (b), we have that $S(1)$ is  $\tau_2$-rigid. Suppose there is some indecomposable $M\in\modu(\Lambda)$ not isomorphic to $S(1)$ such that $M\oplus S(1)$ is $\tau_2$-rigid. By Corollary \ref{CoroNAIR, 3.4} (a), we have that 
$(M\oplus S(1))^{\perp_{\tau_2}}\subseteq M^{\perp_1}\cap M^{\perp_2} \cap S(1)^{\perp_1}\cap S(1)^{\perp_2}.$ In particular, we get that $\Ext^2_\Lambda(M,S(1))=0,$ $\Ext^2_\Lambda(S(1),M)=0$ and $\Ext^1_\Lambda(M,M)=0.$ By the arguments given in \cite[Example 2]{RS89} we conclude that, for the representation 
$$
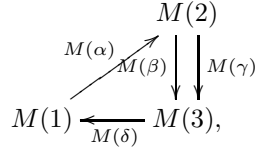
\xymatrix{&M(2)\ar@<1ex>[d]^{M(\gamma)} \ar@<-1ex>[d]_{M(\beta)}\\
M(1)\ar[ru]^{M(\alpha)} & M(3)\ar[l]^{M(\delta)},}$$
$M(\beta)$ is injective and $M(\gamma)$ is surjective. Since $M(\beta)M(\alpha)=0$ and $M(\delta)M(\gamma)=0,$ we 
have that $M(\alpha)=0$ and $M(\delta)=0.$ Thus, by using now that $M$ is indecomposable, we get that $M(1)=0$ and  $M(\beta)$ and  $M(\gamma)$ are isomorphisms. Hence by \cite[Theorem VIII.7.5]{ARS95} it follows that $\Ext^1_\Lambda(M,M)\neq0$ which is a contradiction, proving that $M\oplus S(1)$ can not be $\tau_2$-rigid for any $M\in\modu(\Lambda)$ with $rk(M\oplus S(1))=3.$
\end{ex}

The following example shows an algebra $\Lambda$ having an infinite number of $\tau$-rigid indecomposable $\Lambda$-modules. However it has a finite amount of $\tau_3$-rigid indecomposable $\Lambda$-modules.

\begin{ex} Let $k$ be a field and let $\Lambda$ be the quotient path $k$-algebra, with  radical square zero, associated to the quiver 
$
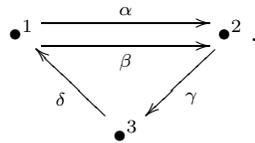
\xymatrix{\bullet^1 \ar@<1ex>[rr]^\alpha \ar@<-1ex>[rr]_\beta&& \bullet^2\ar[dl]^{\gamma}\\
&\bullet^3\ar[ul]^{\delta}}.$ 
\

 For any $p\in\mathbb{P}^1(k)$ and $n\in\mathbb{N},$ we consider the following indecomposable representations:
$$\xymatrix{Q_n:=k^n\ar@<1ex>[rrrr]^{\begin{pmatrix}1_{k^n}\\0\end{pmatrix}}\ar@<-1ex>[rrrr]_{\begin{pmatrix}0\\1_{k^n}\end{pmatrix}}&&&&k^{n+1}\ar[ddll]\\ \\
&&0\ar[uull]},$$
$$\xymatrix{J_n:=k^{n+1}\ar@<1ex>[rrrr]^{(1_{k^n},0)}\ar@<-1ex>[rrrr]_{(0,1_{k^n})}&&&&k^n\ar[ddll]\\ \\
&&0\ar[uull]},$$
$$\xymatrix{R_{p,n}:=k^n\ar@<1ex>[rrrr]^{1_{k^n}}\ar@<-1ex>[rrrr]_{J_{n,p}}&&&&k^n\ar[ddll]\\ \\
&&0\ar[uull]},$$
where $J_{n,p}$ is the Jordan block corresponding to the proper vector $p$ of size $n$. Then, by \cite[Theorem VIII.7.5]{ARS95}, the representative set $\mathrm{ind}(\Lambda)$ of indecomposable $\Lambda$-modules is
$\mathrm{ind}(\Lambda)=\{Q_n,\;J_n,\;R_{p,n+1},\;P(3), P(2), S(3)\;|\;n\in\mathbb{N},\;p\in\mathbb{P}^1(k)\}.$
Moreover, the Auslander-Reiten quiver is given by the following components
$$\tiny{\xymatrix{&&&P(3)\ar[ddr]\\
&I(2)\ar@<1ex>[dr]\ar@<-1ex>[dr]\\
\cdots\;J_2\ar@<1ex>[ur]\ar@<-1ex>[ur]&&S(1)\ar[uur]&&S(3)\ar[ddr]&&S(2)\ar@<1ex>[dr]\ar@<-1ex>[dr]&&Q_2\cdots\\ 
&&&&&&&P(1)\ar@<1ex>[ur]\ar@<-1ex>[ur]\\
&&&&&P(2)\ar[uur]}}$$
$$\xymatrix{R_{p,0}\ar@<1ex>[r]&R_{p,1}\ar@<1ex>[r]\ar@<1ex>[l]&R_{p,2}\ar@<1ex>[l]\ar@<1ex>[r] &\ar@<1ex>[l]\;\cdots&&\forall\;p\in\mathbb{P}^1(k).}$$
By  \cite[Lemma 2.1]{AIR14} and \cite[Theorem VIII.7.5]{ARS95}, we have that $S(2)$ is $\tau$-rigid and $S(2)\oplus X$ is $\tau$-rigid if, and only if, $X\in\{P(2), P(1), P(2)\oplus P(1)\}$. Then, $P(1)$ is $\tau$-rigid and $P(1)\oplus X$ is $\tau$-rigid if, and only if, 
\begin{center}
$X\in\{P(2),P(3),S(2),Q_2,P(2)\oplus P(3), P(2)\oplus S(2),P(3)\oplus Q_2\}.$
\end{center}
On the other hand,  $Q_{n+2}$ is $\tau$-rigid and $Q_{n+2}\oplus X$ is $\tau$-rigid if, and only if, 
\begin{center}
$X\in\{P(3),Q_{n+1},Q_{n+3},P(3)\oplus Q_{n+1},P(3)\oplus Q_{n+3}\}\quad\forall\;n\in\mathbb{N}.$
\end{center}
Furthermore, $S(1)$ is $\tau$-rigid and $S(1)\oplus X$ is $\tau$-rigid if, and only if, 
\begin{center}
$X\in\{I(2),P(3),I(2)\oplus P(3)\}.$
\end{center}
Finally,  $J_{n+1}$ is $\tau$-rigid and $J_{n+1}\oplus X$ is $\tau$-rigid if, and only if, 
\begin{center}
$X\in\{J_{n+2},J_n,P(3),J_{n+2}\oplus P(3),J_n\oplus P(3)\}\quad\forall\;n\in\mathbb{N}.$
\end{center}

Therefore, the support $\tau$-tilting quiver of $\Lambda$ is infinite and has the following shape
$$\begin{turn}{90}\scriptsize
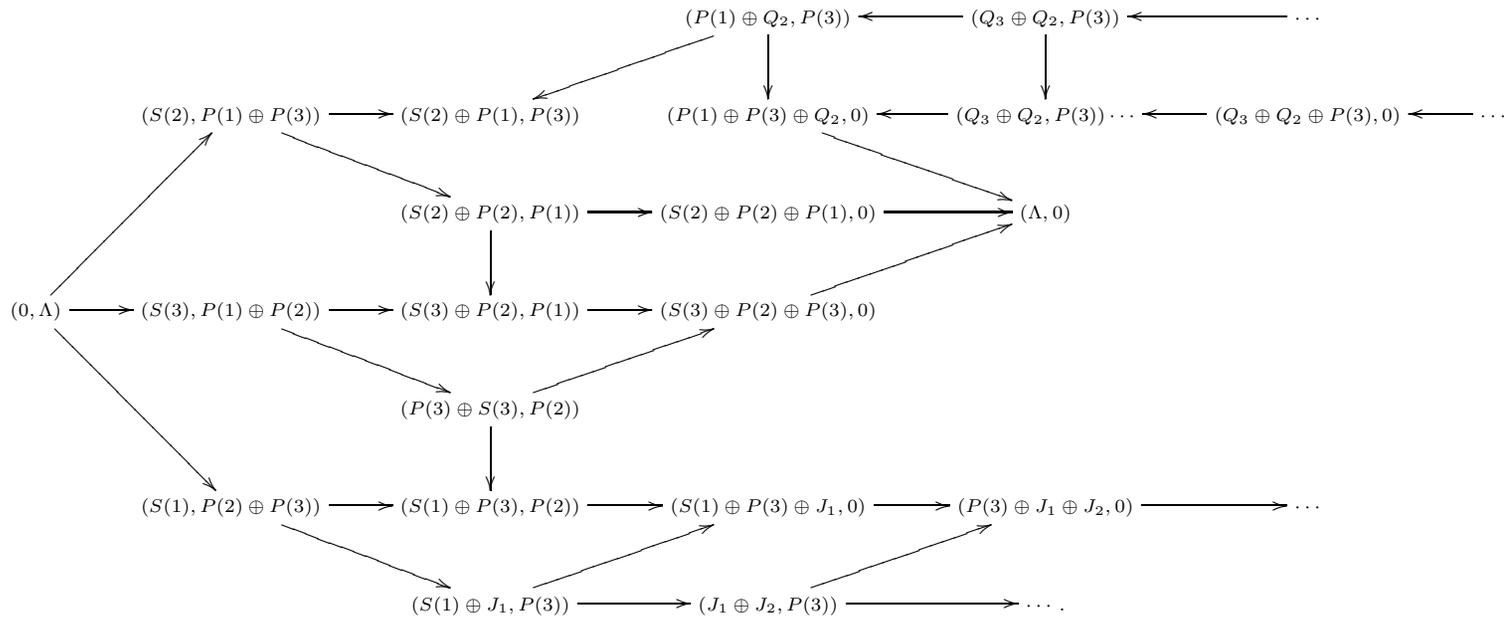
{\xymatrix{&&&(P(1)\oplus Q_2,P(3))\ar[dl]\ar[d]&(Q_3\oplus Q_2,P(3))\ar[l]\ar[d]&\ar[l]\cdots\\
&(S(2), P(1)\oplus P(3))\ar[r]\ar[rd]&(S(2)\oplus P(1), P(3))&(P(1)\oplus P(3)\oplus Q_2, 0)\ar[rd]&(Q_3\oplus Q_2, P(3))\cdots\ar[l]&(Q_3\oplus Q_2\oplus P(3), 0)\ar[l]&\ar[l]\cdots\\
&&(S(2)\oplus P(2), P(1))\ar[r]\ar[d]&(S(2)\oplus P(2)\oplus P(1), 0)\ar[r]&(\Lambda, 0)\\
(0, \Lambda)\ar[ruu]\ar[r]\ar[rdd]&(S(3), P(1)\oplus P(2))\ar[r]\ar[dr]&(S(3)\oplus P(2), P(1))\ar[r]&(S(3)\oplus P(2)\oplus P(3), 0)\ar[ru]\\
&&(P(3)\oplus S(3), P(2))\ar[d]\ar[ru]\\
&(S(1), P(2)\oplus P(3))\ar[r]\ar[rd]&(S(1)\oplus P(3), P(2))\ar[r]&(S(1)\oplus P(3)\oplus J_1, 0)\ar[r]&(P(3)\oplus J_1\oplus J_2, 0)\ar[r]&\cdots\\
&&(S(1)\oplus J_1, P(3))\ar[ru]\ar[r]&(J_1\oplus J_2, P(3))\ar[r]\ar[ru]&\cdots.}}\end{turn}$$
We assert that $P(1), P(2)$ and $ P(3)$ are the unique $\tau_3$-rigid indecomposable $\Lambda$-modules. Let $n\in\mathbb{N}$ and $p\in\mathbb{P}^1(k)$. Consider the following exact sequences
$$0\to S(3)^{n+1}\to P(2)^{n+1}\to P(1)^{n+2}\to Q_{n+2}\to0,$$
$$0\to S(3)^{n+2}\to P(2)^{n+2}\to P(1)^n\to R_{p,n}\to0,$$
$$0\to S(3)^{n+2}\to P(2)^{n+2}\to P(1)^{n+1}\to J_n\to0,$$
$$0\to S(1)\to P(3)\to P(2)\to S(2)\to0,$$
$$0\to S(2)^2\to P(1)\to P(3)\to S(3)\to0.$$
Therefore $\Hom_\Lambda(Q_{n+2},\tau_3(Q_{n+2}))=\Hom_\Lambda(Q_{n+2}, S(1)^{n+1})\neq0$, 
$$\Hom_\Lambda(R_{p,n},\tau_3(R_{p,n}))=\Hom_\Lambda(R_{p,n},S(1)^{n+2})\neq0,$$
$$\Hom_\Lambda(J_n,\tau_3(J_n))=\Hom_\Lambda(J_n,S(1)^{n+2})\neq0,$$ 
$$\Hom_\Lambda(S(2),\tau_3(S(2)))=\Hom_\Lambda(S(2),J_2)\neq0,$$
and $\Hom_\Lambda(S(3),\tau_3(S(3)))=\Hom_\Lambda(S(3),S(3)^2)\neq0$. Thus, our assertion follows.
\end{ex}

The following corollary is an extension of \cite[Proposition 2.5]{AIR14} and it is a nice criterion to detect if one $\Lambda$-module is not $\tau_n$-rigid. In order to prove this result, we recall that for any $X,Y\in\modu(\Lambda),$ we have that $\add(X)\cap\add(Y)=0$ if and only if 
$\Hom_\Lambda(X,Y)=\rad_\Lambda(X,Y).$ 

\begin{cor}\label{NAIR, 2.5} If $M\in\modu(\Lambda)$ is $\tau_n$-rigid, then $\add(P^0(M))\cap\add(P^{-n}(M))=0.$ 
\end{cor}
\begin{proof} Let $M\in\modu(\Lambda)$ be $\tau_n$-rigid. We need to show that $\Hom_\Lambda( P^{-n}(M), P^0(M))$ is equal to $\rad_\Lambda( P^{-n}(M), P^0(M)).$ 
\

Let $f: P^{-n}(M)\to P^0(M)$ be a morphism of $\Lambda$-modules. Then, we can extend it to a morphism of complexes $f^\bullet:\prn(M)\to \prn(M)[n]$ and by Theorem \ref{NAIR, 3.4} we have that $f^\bullet=0$ in $\mathsf{K}^b(\mathrm{proj}(\Lambda)).$ In particular, we get the 
following diagram 
$$\xymatrix{&P^{-n}(M)\ar[d]^f\ar@{-->}[dl]_{\exists\;g}\ar[r]^{\pi^{-n}_M}&P^{1-n}(M)\ar@{-->}[dl]^{\exists\;g'}\\
P^{-1}(M)\ar[r]_{\pi^{-1}_M}&P^0(M)}$$
in $\modu(\Lambda)$, where $f=g'\pi^{-n}_M+\pi^{-1}_M g.$ Since $\pi^{-n}_M\in\mathrm{rad}_\Lambda(P^{-n}(M), P^{1-n}(M))$ and $\pi^{-1}_M\in\mathrm{rad}_\Lambda(P^{-1}(M),P^0(M)),$ we conclude  that $f\in\mathrm{rad}_\Lambda(P^{-n}(M),P^0(M)).$
\end{proof}

\section{$\tau_n$-perpendicular category}

Let $\Lambda$ be an Artin algebra. In this section we study some basic properties of the $\tau_n$-perpendicular category 
$M^{\perp_{\tau_n}},$ for $M\in\modu(\Lambda).$

For the next lemma, we recall that $\mathsf{K}^{b,\leq0}(\proj(\Lambda))$ is the full subcategory of $\mathsf{K}^{b}(\proj(\Lambda))$ whose objects $X^\bullet$ satisfies that $H^k(X^\bullet)=0$ $\forall\,k>0.$

\begin{lem}\label{l6. NAIR, 2.2} For $n,\ell\in\mathbb{N}^+$ and a family of distinguished triangles in $\homo$
$$\{\eta_i: \pr_i\xrightarrow{f_i^\bullet}\prn(M_i)\xrightarrow{g_i^\bullet}\pr_{i+1}\to\pr_i[1]\}_{i=0}^{\ell-1},$$ 
where $\pr_\ell:=\prn(M_\ell)$ and $\pr_0\in\mathsf{K}^{b,\leq0}(\proj(\Lambda)),$  the following statements hold true.
\begin{itemize}
\item[(a)] $H^k(P^\bullet_{\ell-i})=0$ $\quad\forall\,k\in(-n+i,0)$ and $i\in[0,\ell].$
\item[(b)] $P^\bullet_r\in \mathsf{K}^{b,\leq0}(\proj(\Lambda))$ $\quad\forall\,r\in [0,\ell].$
\item[(c)] For $\ell\leq n,$ there exists the following family of exact sequences in $\modu(\Lambda):$
$$H^0(P^\bullet_0)\xrightarrow{H^0(f^\bullet_0)}M_0\xrightarrow{H^0(g^\bullet_0)}H^0(P^\bullet_1)\to 0,$$
$$0\to H^0(\pr_r)\xrightarrow{H^0(f^\bullet_r)} M_r\xrightarrow{H^0(g^\bullet_r)} H^0(\pr_{r+1})\to0\quad \forall\,r\in[1,\ell-2],$$
$$0\to H^0(P^\bullet_{\ell-1})\xrightarrow{H^0(f^\bullet_{\ell-1})}M_{\ell-1}\xrightarrow{H^0(g^\bullet_{\ell-1})}M_\ell\to 0.$$
Moreover $H^0(f^\bullet_0)$ is a monomorphism if $\ell<n.$
\end{itemize}
\end{lem}
\begin{proof} The items (a) and (b) follow by induction using the cohomological functor $H^0:\mathsf{K}^{b}(\proj(\Lambda))\to \modu(\Lambda).$ Finally, the item (c) is a consequence of (a), (b) and the aforementioned  cohomological functor.
\end{proof}

Now, we are ready to state and prove the following variant, for $\tau_n$-tilting modules, of the  Bazzoni's characterization for tilting modules (see Theorem \ref{Ba04, 3.11}).

\begin{teo}\label{p4. NAIR, 2.2} If $M\in\modu(\Lambda)$ is $\tau_n$-tilting, then $M^{\perp_{\tau_n}}=\gen_n(M).$
\end{teo}
\begin{proof} Let $M\in\modu(\Lambda)$ be $\tau_n$-tilting. Let us show that $M^{\perp_{\tau_n}}=\gen_n(M).$
\

Let $X\in M^{\perp_{\tau_n}}.$ By Theorem \ref{NAIR, 3.4} it follows that $\prn(X)\in\prn(M)^{\perp_{>0}}.$ Then,  by \cite[Proposition 2.23 (a)]{AI12}, there is some $\ell\geq 0$ such that 
$$\prn(X)\in\mathcal{M}*\mathcal{M}[1]*\cdots*\mathcal{M}[\ell],$$
 where $\mathcal{M}:=\add(\prn(M))$ in 
$\mathsf{K}^{b}(\proj(\Lambda)).$ If $\ell=0,$ then  $\prn(X)\simeq\prn(M_0)$ for some $M_0\in\add(M),$ and thus 
$X=H^0(\prn(X))\simeq M_0\in \gen_n(M).$ Therefore, we can assume that $\ell\geq 1.$ Hence, by \cite[Lemma 3.8 (a)]{MSSS13}, there exists a 
family of distinguished triangles in $\mathsf{K}^{b}(\proj(\Lambda))$
$$\{\eta_i: \pr_i\to\prn(M_i)\to\pr_{i+1}\to\pr_i[1]\}_{i=0}^{\ell-1},$$  where 
$\pr_{\ell}:=\prn(X),$ $\pr_0=\prn(M'_0)$ and $\{M'_0,M_0,M_1,\cdots, M_{\ell-1}\}\subseteq\add(M).$ In particular, by Lemma \ref{l6. NAIR, 2.2} (b), we have that $\pr_i\in \mathsf{K}^{b,\leq 0}(\proj(\Lambda)),$ for all $i\in[0,\ell].$

Let us consider the following two cases:
\

Case I: Let $\ell\geq n.$ Then, by applying Lemma \ref{l6. NAIR, 2.2} (c) to the family 
$$\{\eta_i: \pr_i\to\prn(M_i)\to\pr_{i+1}\to P^\bullet_i[1]\}_{i={\ell-1-(n-1)}}^{\ell-1},$$ we get the exact sequence  
$M_{\ell-1-(n-1)}\to\cdots\to M_{\ell-1}\to X\to0$ in $\modu(\Lambda)$ and thus $X\in\gen_n(M).$

Case II: Let $\ell<n.$ Then, by Lemma \ref{l6. NAIR, 2.2} (c), there is a exact sequence $0\to M'_0\to M_0\to\cdots\to M_{\ell-1}\to X\to 0.$ Thus 
 $\resdim_{\add(M)}(X)\leq \ell<n$ and therefore  $X\in\gen_n(M)$ since $0\in\add(M).$ Hence we have shown that 
 $M^{\perp_{\tau_n}}\subseteq \gen_n(M).$

Let us prove that $\gen_n(M)\subseteq M^{\perp_{\tau_n}}.$  Indeed, let $Y\in\gen_n(M).$ Then, we have a family of exact sequences 
$\{0\to N_i\to M_i\to N_{i+1}\to0\}_{i=1}^n$ in $\modu(\Lambda),$  with $\{M_i\}_{i=1}^n\subseteq\add(M)$ and $N_{n+1}:=Y.$ Therefore, 
by Proposition \ref{Zi}, there is a family  
$\{\eta_i: \pr(N_i)\to\pr(M_i)\to\pr(N_{i+1})\to\pr(N_i)[1]\}_{i=1}^n$ 
of distinguished triangles in $\mathsf{K}(\mathrm{proj}(\Lambda)).$ Let 
$(X^\bullet,Y^\bullet):=\Hom_{\mathsf{K}(\proj(\Lambda))}(X^\bullet,Y^\bullet).$ For each $i\in[1,n],$ we apply the cohomological functor 
$(\prn(M),-)$ to each $\eta_i,$ and since $(\prn(M),\pr(M_i)[j])=0=(\prn(M),\pr(M_i)[j+1])$ (see Lemma \ref{lemtrunc} and Theorem \ref{NAIR, 3.4}),  we get the following isomorphism
$$(\prn(M),\pr(N_{i+1})[j])\simeq(\prn(M),\pr(N_i)[j+1])\quad\forall\,j\geq 1.$$
Therefore, by Lemma \ref{lemtrunc}, we get the  following isomorphism
$$(\prn(M),\prn(N_{i+1})[j])\simeq(\prn(M),\prn(N_i)[j+1])\quad\forall\,j\geq 1.$$
Then $(\prn(M),\prn(N_{n+1})[j])\simeq(\prn(M),\prn(N_1)[n+j])=0\;\forall\,j\geq 1,$ and thus from
Theorem \ref{NAIR, 3.4}, we conclude that $Y\in M^{\perp_{\tau_n}}.$
\end{proof}

\begin{pro}\label{l1, NHa95} Let $M\in\modu(\Lambda)$ be $\tau_n$-rigid and such that $M^{\perp_{\tau_n}}\subseteq\gen(M).$ Then $\add(M)$ is a relative generator in $M^{\perp_{\tau_n}}.$
\end{pro}
\begin{proof}  Let 
$X\in M^{\perp_{\tau_n}}\subseteq\gen(M).$ Then, there is an exact sequence 
\begin{center}
$0\to Y\to M'\overset{g}{\to}X\to0$ in $\modu(\Lambda),$ 
\end{center}
with 
$g: M'\to X$ an $\add(M)$-precover. We assert that $Y\in M^{\perp_{\tau_n}}.$ 
By Proposition \ref{Zi}, there exists a distinguished triangle 
\begin{center}
 $\eta:\pr(Y)\to\pr(M')\overset{g^\bullet}{\to}\pr(X)\to\pr(Y)[1]$
 \end{center}
 in $\mathsf{K}(\mathrm{proj}(\Lambda)),$ where $H^0(g^\bullet)=g.$ Let $(X^\bullet,Y^\bullet):=\mathrm{Hom}_{\mathsf{K}(\mathrm{proj}(\Lambda))}(X^\bullet,Y^\bullet).$ By applying the functor $(\prn(M),-)$ to $\eta$ and using Theorem \ref{NAIR, 3.4} and Lemma \ref{lemtrunc},  we get the following exact sequences $(i)$ and $(ii),$ $\forall$ $k\geq 1:$
$$\footnotesize{(i)\xymatrix{(\prn(M),\pr(M'))\xrightarrow{(\prn(M),g^\bullet)}(\prn(M),\pr(X))\to(\prn(M),\pr(Y)[1])\to0,}}$$
$$\footnotesize{(ii)\xymatrix{0=(\prn(M),\pr(X)[k])\to(\prn(M),\pr(Y)[k+1])\to(\prn(M),\pr(M')[k+1])=0.}}$$
Let us show that  $(\prn(M),g^\bullet)$ is an epimorphism. Consider $t^\bullet: \prn(M)\to\pr(X)$ in $\mathsf{K}(\mathrm{proj}(\Lambda)).$ 
 Then, there is a morphism $t:M\to X$ in $\modu(\Lambda),$ where $H^0(t^\bullet)=t;$ and using that $g: M'\to X$ is a $\add(M)$-precover, 
we get some $r:M\to M'$ such that $t=gr.$ By the comparison lemma in homological algebra, there is some $r^\bullet: \prn(M)\to \pr(M')$ in 
$\mathsf{K}(\mathrm{proj}(\Lambda))$ such that $g^\bullet r^\bullet  =t^\bullet.$ Thus,  $(\prn(M),g^\bullet)$ is an epimorphism, and then, 
by (i), we conclude that $(\prn(M),\pr(Y)[1])=0.$ Then, by using (ii), we get that $(\prn(M),\pr(Y)[k])=0$ $\forall\,k\geq 1.$ Finally, from 
Theorem \ref{NAIR, 3.4} and Lemma \ref{lemtrunc}, we conclude that $Y\in M^{\perp_{\tau_n}}.$ 
\end{proof}

\section{$\tau_n$-tilting and $n$-tilting modules}

The main objective of this section is to give some basic properties of the $\tau_n$-tilting modules in order to show equivalent conditions when these are $n$-tilting. In this section $\Lambda$ will denote an Artin algebra.

\begin{lem}\label{l5. NAIR, 2.2} Let $M\in\modu(\Lambda)$ be faithful and $\tau_n$-rigid. Then $\pd(M)\leq n$ and $M\in M^\perp$. 
\end{lem}
\begin{proof}   Since $M$ is faithful, there is an exact sequence $0\to K\to M'\to\mathrm{D_{\Lambda^{op}}}(\Lambda_\Lambda)\to0$ in $\mathrm{mod}(\Lambda)$, with $M'\in\mathrm{add}(M)$. Applying the functor $\Hom_\Lambda(-,\tau_n(M))$ to the previous exact sequence, we have the exact sequence
\begin{center}
$0\to\Hom_\Lambda(\mathrm{D_{\Lambda^{op}}}(\Lambda_\Lambda),\tau_n(M))\to\Hom_\Lambda(M',\tau_n(M))=0.$
\end{center}
Thus, by \cite[Lemma IV.2.7]{ASS06}, it follows that $\pd(M)\leq n.$ Therefore, by Corollary \ref{CoroNAIR, 3.4}, we conclude that $M\in M^\perp.$ \end{proof}

\begin{lem}\label{LemHsim} Let $n\in\mathbb{Z},$ $P\in\proj(\Lambda)$ and $X^\bullet\in\Ksf(\modu(\Lambda)).$ Then 
\begin{center}
$\Hom_{\Ksf(\modu(\Lambda))}(P[0],X^\bullet[n])\simeq\Hom_\Lambda(P,H^n(X^\bullet)). $
\end{center}
\end{lem}
\begin{proof} It is straightforward to show that 
\begin{center}
$\Hom_{\Ksf(\modu(\Lambda))}(P[0],X^\bullet[n])\simeq H^n(\Hom_\Lambda(P,X^\bullet)). $ 
\end{center}
Finally, we have that $H^n(\Hom_\Lambda(P,X^\bullet))\simeq \Hom_\Lambda(P,H^n(X^\bullet))$ since $\Hom_\Lambda(P,-)$ is an exact functor.
\end{proof}

Let $n\leq m$ be integers.  We denote by $\mathsf{C}^{[n,m]}(\proj(\Lambda))$ the subcategory of $\mathsf{C}(\proj(\Lambda))$ whose objects are all the $P^\bullet\in\mathsf{C}(\proj(\Lambda))$ such that $P^i=0$ $\forall$ $i<n$ and $i>m.$ 

\begin{lem}\label{l 2.1.1} For $n\in\mathbb{N}^+$ and $P^\bullet\in\mathsf{C}^{[-n-1,0]}(\proj(\Lambda))$ presilting in $\homo,$ the following statements hold true.
\begin{itemize}
\item[(a)] $\HomK(P^\bullet_{\geq-n},P^\bullet_{\geq-n}[m])\simeq \Hom_\Lambda(P^{-n-1},H^{m-n}(P^\bullet))\;\;\forall\,m\geq 1.$
\item[(b)] $P^\bullet_{\geq-n}$ is presilting in $\homo$ $\Leftrightarrow$ 
$\Hom_\Lambda(P^{-n-1},H^{k}(P^\bullet))=0\quad\forall\;k\in[-n+1,0].$
\end{itemize}
\end{lem}
\begin{proof}  Note that (b) is a consequence of (a). Thus, it is enough to show (a).
\

 Let $m\geq 1.$ Consider the following distinguished triangle  in $\homo$
\begin{center}
$\eta:\; P^\bullet[-1]\to P^{-n-1}[n]\to P^\bullet_{\geq-n}\to P^\bullet.$
\end{center}
 Applying the functor $\HomK(-,P^\bullet[m])$ to $\eta$ and using that $P^\bullet$ is presilting,  we get the following exact sequence 
$$0\to\HomK(P^\bullet_{\geq-n},P^\bullet[m])\to\HomK(P^{-n-1}[n],P^\bullet[m])\to0.$$
Thus, by Lemma \ref{lemtrunc}, we have that 
$$\HomK(P^\bullet_{\geq-n},P^\bullet_{\geq-n}[m])\simeq\HomK(P^{-n-1}[n],P^\bullet[m]).$$
Finally, from Lemma \ref{LemHsim}, we get that 
$$\HomK(P^{-n-1}[n],P^\bullet[m])\simeq \Hom_\Lambda(P^{-n-1}, H^{m-n}(P^\bullet)),$$
proving the result.
\end{proof}

Recall that $M\in\modu(\Lambda)$ is $\mathbf{sincere}$ if every simple $\Lambda$-module appears as a composition factor in $M$. The following result gives us necessary and sufficient conditions for a $\tau_{n+1}$-rigid $\Lambda$-module to be $\tau_{n}$-rigid.

\begin{pro}\label{c 2.1.1} For a  $\tau_{n+1}$-rigid $M\in\mathrm{mod}(\Lambda),$ with $n\geq 1,$  the following statements hold true.
\begin{itemize}
\item[(a)] $M$ is $\tau_n$-rigid in $\modu(\Lambda)$ $\Leftrightarrow$ $\Hom_\Lambda(P^{-n-1}(M),M)=0.$
\item[(b)]  If $M$ is sincere and $\pd(M)> n,$ then $M$ is not $\tau_n$-rigid in $\modu(\Lambda).$  
\end{itemize}
\end{pro}
\begin{proof} (a) It follows from Lemma \ref{l 2.1.1} (b) and Theorem \ref{NAIR, 3.4}.
\

(b) Note that $P^{-n-1}(M)\neq 0$ since $\pd(M)> n.$ Using that $M$ is sincere and  $P^{-n-1}(M)\neq 0,$ we get that $\Hom_\Lambda(P^{-n-1}(M),M)\neq 0.$ Thus, from (a), we conclude that $M$ is not $\tau_n$-rigid in $\modu(\Lambda).$
\end{proof}

\begin{pro}\label{p1. NAIR, 2.2} Let $n\in\mathbb{N}^+$ and $P^\bullet\in\mathsf{C}^{[-n,0]}(\proj(\Lambda))$ with $H^m(P^\bullet)=0$ $\forall$ $m\in(-n,0)$ and let $M:=H^0(P^\bullet).$ Then, there exists $Q\in\proj(\Lambda)$ such that 
\begin{center}
$P^\bullet=\prn(M)\coprod Q[n]$ in $\homo.$ 
\end{center}
\end{pro}
\begin{proof} Since $\prn(M)\to M$ is the minimal projective $n$-presentation of $M$ and $P^\bullet\to M$ is another projective $n$-presentation of $M,$ 
it can be shown that there exist $Q_0,Q_1,\cdots, Q_n$ in $\proj(\Lambda)$ such that 
\begin{center}
$P^\bullet=\prn(M)\coprod Q[n]\coprod\coprod_{i=0}^{n-1}(Q_i\xrightarrow{1}Q_i)[i]$ in 
$\mathsf{C}(\proj(\Lambda)).$
\end{center}
From the above decomposition of complexes, we get the result.
\end{proof}

\begin{pro}\label{p1} $rk(M)=rk(\prn(M)),$ for any $M\in\modu(\Lambda).$
\end{pro}
\begin{proof} Since $\prn(\coprod_{i=1}^n\,M_i)=\coprod_{i=1}^n\,\prn(M_i)$ (use that the finite coproduct of projective covers is a projective cover) and $H^0(\prn(N))=N,$ we can assume that $M$ is indecomposable. Thus, we need to show that $\prn(M)$ is indecomposable in $\homo.$
\

Let $\prn(M)=P^\bullet\coprod Q^\bullet$ in $\homo.$ Then $H^0(P^\bullet)\coprod H^0(Q^\bullet)=M$ and since $M$ is indecomposable, we can assume that $H^0(P^\bullet)=M$. Using that $H^k(P^\bullet)=0$ $\forall$ $k\in[-n+1,-1],$ we get from Proposition \ref{p1. NAIR, 2.2} that $\prn(M)\;|\;P^\bullet$ in $\homo.$ Since $P^\bullet\;|\;\prn(M)$ and $\homo$ is a Krull-Schmidt category, we conclude that $\prn(M)\cong P^\bullet$ and thus $Q^\bullet=0.$ Therefore $\prn(M)$ is indecomposable in $\homo.$\end{proof}

Now, we are ready to state and proof the following basic properties of $\tau_n$-tilting modules.

\begin{teo}\label{c1} For a $\tau_n$-tilting $M\in\modu(\Lambda),$ the following statements hold true.
\begin{itemize}
\item[(a)] $M$ is sincere and $rk(M)=rk(\Lambda).$
\item[(b)] $M$ is not $\tau_{n+1}$-rigid in $\modu(\Lambda)$ or $\pd(M)\leq n.$
\end{itemize}
\end{teo}
\begin{proof}  (a) Since $\Lambda[0]$ is silting in $\homo,$ we get from \cite[Corollary, 2.28]{AI12} and Proposition \ref{p1} that $rk(M)=rk(\prn(M))=rk(\Lambda[0])=rk(\Lambda).$
\

Let us prove that $M$ is sincere. Let $P\in\proj(\Lambda)$ be such that $\Hom_\Lambda(P,M)=0.$ We need to show that $P=0.$ Consider $Q^\bullet:=\prn(M)\coprod P[n].$ 
We assert that $Q^\bullet$ is silting in $\homo.$ Since $\prn(M)$ is silting and a direct summand of $Q^\bullet,$ it is enough to prove that $Q^\bullet$ is presilting.
Indeed, for $k\geq 1,$ we have that 
$\Hom_{\homo}(Q^\bullet,Q^\bullet[k])$ decomposes as the coproduct of four direct summands, three of them are easy to check that they are zero. The only one that needs more details is the direct summand $\Hom_{\homo}(P[n],\prn(M)[k]).$ In order to show that this summand is also zero, by Lemma \ref{LemHsim}, it is enough to show that 
$\Hom_\Lambda (P,H^{k-n}(\prn(M)))=0,$ which holds true since  $H^{k-n}(\prn(M))=0$ for 
$k<n-1$ or $k>n,$ and $H^{k-n}(\prn(M))=M$ for $k=n.$ Thus we get that $Q^\bullet$ is silting in $\homo.$ Then, by \cite[Corollary, 2.28]{AI12}, we conclude that 
\begin{center}
$rk(\prn(M))=rk(Q^\bullet)=rk(\prn(M)\coprod P[n])=rk(\prn(M))+rk(P[n]).$
\end{center}
Therefore $rk(P[n])=0$ and thus $P=0.$
\

(b) By (a) we know that $M$ is sincere. Suppose that $M$ is $\tau_{n+1}$-rigid and 
$\pd(M)> n.$ Then, by Proposition \ref{c 2.1.1} (b), we get that $M$ is not $\tau_n$-rigid, contradicting that $M$ is $\tau_n$-tilting.
\end{proof}

The following result gives us \cite[Theorem 1.3 and Theorem 3.2]{Zh22}  as a particular case (take $n=1$) and characterizes the $\tau_n$-tilting modules which are $n$-tilting.

\begin{teo}\label{TEO} For a $\tau_n$-tilting $M\in\modu(\Lambda),$ the following statements are equivalent.
\begin{itemize}
\item[(a)] $M$ is $n$-tilting in $\modu(\Lambda).$
\item[(b)] $\coresdim_{\add(M)}(\Lambda)\leq n.$
\item[(c)] $M$ is faithful.
\item[(d)] $M\in{}^{\perp_{n+i}}\gen(M)$ $\forall$ $i\geq 0.$
\item[(e)] $M\in{}^{\perp_{n+1}}\gen(M).$
\item[(f)] $M$ is $\tau_{n+1}$-rigid.
\item[(g)] $\pd(M)\leq n.$
\item[(h)] $M\in M^{\perp_{n+1}}.$
\end{itemize}
\end{teo}

\begin{proof} The proof of (a) $\Rightarrow$ (b,h) and (d) $\Rightarrow$ (e) are trivial.

(b) $\Rightarrow$ (c) It is follows from the \cite[Lemma VI.2.2]{ASS06}.

(c) $\Rightarrow$ (d) By Lemma \ref{l5. NAIR, 2.2}, we get that $\pd(M)\leq n,$ and then $M\in {}^{\perp_{n+i}}\gen(M)$ $\forall$ $i\geq 1.$ On the other hand, from Corollary \ref{CoroNAIR, 3.4} (c), we know that $M\in {}^{\perp_{n}}\gen(M).$ 
\

(e) $\Rightarrow$ (f) By Corollary \ref{CoroNAIR, 3.4} (c), we get 
$M\in {}^{\perp_{0}}\tau_{n+1}(M).$ Moreover, since 
$M\in M^{\perp_{\tau_n}}={}^{\perp_{0}}\tau_{n}(M)\cap\cap_{i=1}^{n-1}\,M^{\perp_i},$ from Corollary \ref{CoroNAIR, 3.4} (a), it follows that $M\in\cap_{i=1}^{n}\,M^{\perp_i}.$ Therefore, $M\in M^{\perp_{\tau_{n+1}}}.$
\

(f) $\Rightarrow$ (g) It follows from Theorem \ref{c1} (b).
\

(g) $\Rightarrow$ (a) Note that $\prn(M)=P^\bullet(M)$ since $\pd(M)\leq n.$  In particular, we get that  $P^\bullet(M)$ is silting in $\homo$ and then, from  Proposition \ref{We13, 3.7}, we conclude that $M$ is $n$-tilting.
\

(h) $\Rightarrow$ (g) Notice firstly that $M\in \cap_{i=1}^{n+1}M^{\perp_i}$ since $M$ is $\tau_n$-rigid and $M\in M^{\perp_{n+1}}.$

Suppose that $\pd(M)>n.$ Then, by Theorem \ref{c1} (b) $M$ is not $\tau_{n+1}$-rigid and thus $M\not\in {}^{\perp_0}\tau_{n+1}(M)$ since $M\in \cap_{i=1}^{n+1}M^{\perp_i}.$ Moreover, from Corollary \ref{CoroNAIR, 3.4} (c), we get 
that $M\not\in{}^{\perp_{n+1}}\gen(M)\subseteq {}^{\perp_{n+1}}M$ which is a contradiction, proving that $\pd(M)\leq n.$
\end{proof}

\section{ $\tau_{n,m}$-tilting modules}

The main objective of this section is to study $\tau_{n,m}$-tilting $\Lambda$-modules and its relationship with silting objects in $\homo$ and $m$-tilting 
$\Lambda/I$-modules, where $I\trianglelefteq\Lambda$ comes as the annihilator of sincere modules. The above is a natural extension of the connection, given  in \cite{Ja15}, between $\tau$-tilting 
$\Lambda$-modules  and $1$-tilting $\Lambda/I$-modules; and the given  one in \cite{AIR14} between $\tau$-tilting $\Lambda$-modules  and $2$-term silting complexes in $\homo$.
\

We start with a very useful characterization of the sincere $\Lambda$-modules.

\begin{lem}\label{l4. NAIR, 2.2} Let $M\in\modu(\Lambda).$ Then $M$ is sincere $\Leftrightarrow$  $^{\perp_0}\gen(M)=0.$
\end{lem}
\begin{proof}
$(\Rightarrow)$ Let $N\in{}^{\perp_0}\gen(M).$ We assert that $\Hom_\Lambda(S,N)=0,$ for any simple $\Lambda$-module $S.$ Indeed, let 
$0= M_0\leq M_1\leq\cdots\leq M_n=M$ be a composition series for $M,$ and let $S$ be a simple $\Lambda$-module. Since  $M$ is sincere, there is some $j\in[0,n-1]$ such that $S\cong M_{j+1}/M_j.$ Thus, we get the following exact and commutative diagram
$$\xymatrix{0\ar[r]& M_j\ar[r]\ar@{=}[d]& M_{j+1}\ar @{^{(}->}[d]\ar[r]& S\ar[r]\ar[d]_{t_j}&0\\
0\ar[r]& M_{j}\ar[r]& M\ar[r]& M/M_j\ar[r]&0.}$$
By  snake's Lemma, $t_j$ is a monomorphism and thus 
\begin{center}
$0\to\Hom_\Lambda(N,S)\to\Hom_\Lambda(N,M/M_j)$
\end{center} 
is an exact sequence. Note that  $\Hom_\Lambda(N,M/M_j)=0$  since 
$M/M_j\in \gen(M)$ and therefore our assertion follows. Finally, using that $N$ is of finite length, from our assertion, we get that $N=0.$
\

$(\Leftarrow)$ Let $P\in\proj(\Lambda)$ be such that $\Hom_\Lambda(P,M)=0.$ Using that $P\in\proj(\Lambda),$ we get that  $P\in{}^{\perp_0}\gen(M)$ and then $P=0.$
\end{proof}

\begin{lem}\label{l3. NAIR, 2.2} 
Let $P^\bullet\in\mathsf{C}^{\leq0}(\proj(\Lambda))$ and $M\in\modu(\Lambda).$ Then 
\begin{center}
$\Hom_\Lambda(H^0(P^\bullet),M)=0\;$ $\Leftrightarrow$ $\;\Hom_{\mathsf{K}(\proj(\Lambda))}(P^\bullet,\pr(M))=0$.
\end{center}
\end{lem}
\begin{proof} ($\Rightarrow$) Let $f^\bullet: P^\bullet\to\pr(M)$ be a morphism in $\mathsf{C}(\proj(\Lambda)).$ Then, there exists $f: H^0(P^\bullet)\to M$ such that the following diagram commutes
$$\xymatrix{P^{-1}\ar[r]^{d^{-1}_{P^\bullet}}\ar[d]_{f^{-1}}&P^0\ar[r]^p\ar[d]^{f^0}&H^0(P^\bullet)\ar[r]\ar[d]^f&0\\
P^{-1}(M)\ar[r]_{\pi^{-1}_M}&P^0(M)\ar[r]_{\pi^0_M}&M\ar[r]&0.}$$
Since $\Hom_\Lambda(H^0(P^\bullet),M)=0,$ and thus $\pi^0_Mf^0=0,$ there exists 
$t: P^0\to P^{-1}(M)$ such that $\pi^{-1}_Mt=f^0$. Using that $H^k(\pr(M))=0$ $\forall$ $k<-1,$ it follows from the Lemma \ref{l1, NAIR, 3.4} (a) that $f^\bullet=0$ in 
$\mathsf{K}(\proj(\Lambda))$.

($\Leftarrow$) Let $f: H^0(P^\bullet)\to M$ be a morphism in $\modu(\Lambda).$ Then, we can extend it to a morphism $f^\bullet: P^\bullet\to\pr(M)$ in $\mathsf{C}(\proj(\Lambda))$ such that the following diagram is commutative
$$\xymatrix{\cdots\ar[r]&P^{-1}\ar[r]^{d^{-1}_{P^\bullet}}\ar[d]_{f^{-1}}&P^0\ar[r]^p\ar[d]^{f^0}&H^0(P^\bullet)\ar[r]\ar[d]^f&0\\
\cdots\ar[r]&P^{-1}(M)\ar[r]_{\pi^{-1}_M}&P^0(M)\ar[r]_{\pi^0_M}&M\ar[r]&0.}$$
Since $\Hom_{\mathsf{K}(\proj(\Lambda))}(P^\bullet,\pr(M))=0$, there is $t: P^0\to P^{-1}(M)$ such that $\pi^{-1}_Mt=f^0$ and so $0=\pi^0_M\pi^{-1}_Mt=\pi^0_Mf^0=fp$. Thus $f=0$ since $p$ is an epimorphism.
\end{proof}

\begin{lem}\label{l1. NAIR, 2.2} Let $m\leq k\leq n$ and $P^\bullet\in\mathsf{C}^{[m,n]}(\proj(\Lambda))$ such that $H^i(P^\bullet)=0$ $\forall$ $i\in[k+1, n].$ Then, there is $Q^\bullet\in\mathsf{C}^{[m,k]}(\proj(\Lambda))$ such that $P^\bullet\simeq Q^\bullet$ in $\homo.$
\end{lem}
\begin{proof} We will carry on the proof by reverse induction on $k\in[m,n].$ For $k=n,$ we just take $Q^\bullet:=P^\bullet.$
\

Let $k<n.$ Suppose that the lemma is true for $k,$ and then, we show that it is true for $k-1.$ Let $H^i(P^\bullet)=0$ $\forall$ $i\in[k, n].$ Then, by the inductive hypothesis, there is some $T^\bullet\in\mathsf{C}^{[m,k]}(\proj(\Lambda))$ such that $P^\bullet\simeq T^\bullet$ in $\homo.$ In particular, $H^{k}(T^\bullet)=H^{k}(P^\bullet)=0$ and thus $d_{T^\bullet}^{k-1}:T^{k-1}\to T^k$ is an split-epi. Therefore 
$T^{k-1}=T^k\coprod Q^{k-1},$ $d_{T^\bullet}^{k-1}=\begin{pmatrix} 1_{T^k} & 0\end{pmatrix}$ and  
$d_{T^\bullet}^{k-2}=\begin{pmatrix} 0\\d\end{pmatrix}$ since $d_{T^\bullet}^{k-1}d_{T^\bullet}^{k-2}=0.$ Consider the complex 
$$Q^\bullet:\;\cdots\to T^{k-3}\xrightarrow{d_{T^\bullet}^{k-3}}T^{k-2}\xrightarrow{d}Q^{k-1}\to 0\to\cdots.$$
Then, we have that $Q^\bullet\in\mathsf{C}^{[m,k-1]}(\proj(\Lambda))$ and $T^\bullet=(T^k\xrightarrow{1_{T^k}} T^k)[-k]\coprod Q^\bullet.$
As a consequence, we get that $P^\bullet\simeq Q^\bullet$ in $\homo$ since the cochain complex $(T^k\xrightarrow{1_{T^k}} T^k)[-k]$ is  zero in the homotopic category. 
\end{proof}

 For $M\in\modu(\Lambda),$ we recall that 
$\Gamma_{\!\!M}:=\Lambda/\ann(M).$ Note that $M\in\modu(\Gamma_{\!\!M})$ and then $\gen(M)\subseteq \modu(\Gamma_{\!\!M}).$ Since 
$\modu(\Gamma_{\!\!M})\subseteq\modu(\Lambda),$ for $\X\subseteq \modu(\Gamma_{\!\!M}),$ the notation of the classes $\X^\perp$ and $\X^{\perp_{\tau_n}}$ is ambiguous because it is not clear where they are computed. In order to avoid confusion, we write $\X^{\perp_{\Gamma_{\!\!M}}}$ and $\X^{\perp_{\tau_{n,\Gamma_{\!\!M}}}}$ to emphasize that these orthogonal categories are computed in $\modu(\Gamma_{\!\!M}),$ otherwise it means they are computed in $\modu(\Lambda).$
\

The following result can be understood as a generalization of Lemma 
\ref{l4. NAIR, 2.2}.

\begin{pro}\label{p3. NAIR, 2.2} Let $M$ be sincere in $\modu(\Lambda)$ and $n$-tilting in $\modu(\Gamma_{\!\!M}),$ and let $P^\bullet\in\mathsf{C}^{[-s,0]}(\proj(\Lambda)),$ for $s\geq 0.$ If $P^\bullet\in{}^{\perp_{\geq0}}(\pr(M))$ in 
$\Ksf(\proj(\Lambda)),$ then $P^\bullet=0$ in $\Ksf(\proj(\Lambda)).$
\end{pro}
\begin{proof} We start by proving the following assertions $(i)$ and $(ii).$
\

$(i)$ $\Hom_{\mathsf{K}(\proj(\Lambda))}(\pr,\pr(X))=0\;$ $\forall$ $X\in M^{\perp_{\Gamma_{\!\!M}}}.$\\
Indeed, let $X\in M^{\perp_{\Gamma_{\!\!M}}}.$ Since $M$ is $n$-tilting in 
$\modu(\Gamma_{\!\!M}),$ by Theorem \ref{AR91, 5.5} (b) there is an exact sequence $\cdots\to M_1\overset{f_1}{\to}M_0\overset{f_0}{\to}X\to0$ in $\modu(\Gamma_M),$ where $\{M_j\}_{j=0}^\infty$ in $\add(M).$ In particular, for each 
$k\in\mathbb{N},$ there is the exact sequence $0\to \Ima(f_{k+1})\to M_k\to \Ima(f_k)\to 0$ in $\modu(\Lambda).$ Hence, by Proposition \ref{Zi}, we get the family of distinguished triangles in $\mathsf{K}(\proj(\Lambda))$
\begin{center}
$\{\eta_k: \pr(\mathrm{Im}(f_{k+1}))\to\pr(M_k)\to\pr(\mathrm{Im}(f_k))\to\pr(\mathrm{Im}(f_{k+1}))[1]\}_{k=0}^\infty.$
\end{center}
  Applying $\mathrm{Hom}_{\mathsf{K}(\proj(\Lambda))}(P^\bullet,-)$ to 
  each $\eta_k$ and using that $P^\bullet\in{}^{\perp_{\geq0}}(\pr(M))$ in 
$\Ksf(\proj(\Lambda)),$ we get for all $k\geq 0$ that 
$$\small{\Hom_{\mathsf{K}(\proj(\Lambda))}(P^\bullet,\pr(\mathrm{Im}(f_k))[k])\simeq\Hom_{\mathsf{K}(\proj(\Lambda))}(P^\bullet,\pr(\mathrm{Im}(f_{k+1}))[k+1])}.$$
 Note that $P^h=0$ for $h<-s$ and then $\mathrm{Hom}_{\mathsf{K}(\proj(\Lambda))}(P^\bullet,\pr(\mathrm{Im}(f_{s+1}))[s+1])=0.$ Therefore  $\mathrm{Hom}_{\mathsf{K}(\proj(\Lambda))}(P^\bullet,\pr(X))=0$ since $\Ima(f_0)=X.$
 \

$(ii)$ $H^0(\pr)=0.$\\ 
Indeed,  let $Z\in\gen(M)$. Since $M$ is $n$-tilting in 
$\modu(\Gamma_{\!\!M}),$ we get by Theorem \ref{AR91, 5.5} (a)  that the class $M^{\perp_{\Gamma_{\!\!M}}}$ is special preenveloping. Hence, there is an exact sequence $\eta:\,0\to Z\to Y\to W\to 0$ in $\modu(\Gamma_{\!\!M}),$ where 
$Y\in M^{\perp_{\Gamma_{\!\!M}}}.$  Using that $\eta$ is also an exact sequence in $\modu(\Lambda),$ by Proposition \ref{Zi} we have the distinguished triangle in $\mathsf{K}(\proj(\Lambda))$
\begin{center}
$\overline{\eta}:\;\pr(Z)\to\pr(Y)\to\pr(W)\to\pr(Z)[1].$
\end{center} 
Note that $H^1(P^\bullet)=0$ and then by Lemma \ref{l3. NAIR, 2.2}, 
$\Hom_{\Ksf(\proj(\Lambda))}(P^\bullet,P^\bullet(W)[-1])=0.$ Then, by 
applying $\mathrm{Hom}_{\mathsf{K}(\proj(\Lambda))}(P^\bullet,-)$ to 
$\overline{\eta},$  we have the exact sequence 
$$0\to\mathrm{Hom}_{\mathsf{K}(\proj(\Lambda))}(P^\bullet,\pr(Z))\to\mathrm{Hom}_{\mathsf{K}(\proj(\Lambda))}(P^\bullet,\pr(Y))$$
and then by $(ii)$ $\mathrm{Hom}_{\mathsf{K}(\proj(\Lambda))}(P^\bullet,\pr(Z))=0$ since $Y\in M^{\perp_{\Gamma_{\!\!M}}}.$  Thus by Lemma \ref{l3. NAIR, 2.2} $\mathrm{Hom}_\Lambda(H^0(P^\bullet),Z)=0$ for any $Z\in\gen(M).$ Finally, from Lemma \ref{l4. NAIR, 2.2} we conclude that $H^0(P^\bullet)=0$ since $M$ is sincere.
\

Let us show now that $P^\bullet=0$ in $\homo.$ Indeed, for $(ii),$  $H^0(P^\bullet)=0$ and 
then by  Lemma \ref{l1. NAIR, 2.2} there exists $Q^\bullet_1\in\mathsf{C}^{[-s,-1]}(\proj(\Lambda))$ such that $P^\bullet\simeq Q^\bullet_1$ in $\homo.$ In particular 
$Q^\bullet_1[-1]\in{}^{\perp_{\geq0}}(\pr(M))$ in 
$\Ksf(\proj(\Lambda)).$ Thus, we can apply $(ii)$ to $Q^\bullet_1[-1]$ and then 
by  Lemma \ref{l1. NAIR, 2.2} we get some $Q^\bullet_2\in\mathsf{C}^{[-s,-1]}(\proj(\Lambda))$ such that 
$Q^\bullet_1[-1]\simeq Q^\bullet_2$ in $\homo.$ Therefore $P^\bullet\simeq Q^\bullet_2[1]$ in $\homo,$ with $Q^\bullet_2[1]\in\mathsf{C}^{[-s,-2]}(\proj(\Lambda)).$ Note that we can repeat this procedure and hence we get that 
$P^\bullet\simeq Q^\bullet_s[s-1]$ in $\homo.$ with $Q^\bullet_s[s-1]\in\mathsf{C}^{[-s,-s]}(\proj(\Lambda)).$ Therefore, there is some $Q\in \proj(\Lambda)$ such that $P^\bullet\simeq Q[s]$ in $\homo.$ In particular 
$$0=\Hom_{\Ksf(\proj(\Lambda))}(P^\bullet,\pr(M)[s])\simeq \Hom_{\Ksf(\proj(\Lambda))}(Q[0],\pr(M)).$$
Then, by Lemma \ref{LemHsim} $\Hom_\Lambda(Q,M)=0$ and thus $Q=0$ since $M$ is sincere. Finally, $P^\bullet =0$ in $\homo$ since $Q=0$ in $\modu(\Lambda).$
\end{proof}

\begin{lem}\label{l2NAIR, 2.2} Let $\pr\in\homo$ be a presilting object. Then, for a  family of distinguished triangles $\{\eta_i: Q_i^\bullet\xrightarrow{f_i^\bullet}\pr_i\to Q_{i+1}^\bullet\to Q_i^\bullet[1]\}_{i=1}^n$ in $\homo,$ with $\{\pr_i\}_{i=1}^n\subseteq\add(\pr)$ and $Q^\bullet_{n+1}\in\add(\pr),$ the following statements hold true.
\begin{itemize}
\item[(a)] $Q^\bullet_i\in{}^{\perp_{>0}}\pr\;$ $\forall$ $i\in[1,n].$
\item[(b)] $f_i^\bullet:Q_i^\bullet\to P_i^\bullet$ is an $\add(\pr)$-preenvelope $\;\forall$ $i\in[1,n].$\end{itemize}
\end{lem}
\begin{proof} (a) It can be shown by using reverse induction on $i\in[1,n].$
\

(b) It follows from (a), by applying the functor $\HomK(-,Q^\bullet),$ with $Q^\bullet\in\add(\pr),$ to each triangle $\eta_i,$ and then using that $\pr$ is presilting in $\homo.$
\end{proof}

\begin{pro}\label{p1p} For $M\in\modu(\Lambda),$ $M'\in\add(M)$ and the distinguished triangle $\eta:\;\Lambda[0]\xrightarrow{f^\bullet}\prn(M')\to Q^\bullet\to\Lambda[1]$ in $\homo,$ the following statements hold true.
\begin{itemize}
\item[(a)] $f^\bullet:\Lambda[0]\to \prn(M')$ is an $\add(\prn(M))$-preenvelope if, and only if, $H^0(f^\bullet): \Lambda\to M'$ is an $\add(M)$-preenvelope.
\item[(b)] If $H^0(f^\bullet): \Lambda\to M'$ is an $\add(M)$-preenvelope and $M$ is $\tau_n$-rigid, then $Q^\bullet\in{}^{\perp_{>0}}\prn(M).$
\end{itemize}
\end{pro}
\begin{proof}(a) $(\Rightarrow)$ Let $g: \Lambda\to M''$ be in $\modu(\Lambda),$ with $M''\in\add(M).$ By Lemma \ref{LemHsim} there is $g^\bullet: \Lambda[0]\to\prn(M'')$ in $\homo$ such that $H^0(g^\bullet)=g.$ Since $f^\bullet$ is an $\add(\prn(M))$-preenvelope, there is $k^\bullet: \prn(M')\to\prn(M'')$ in $\homo$ such that $g^\bullet=k^\bullet f^\bullet$ and thus $g=H^0(k^\bullet)H^0(f^\bullet).$ 

$(\Leftarrow)$ Let $g^\bullet: \Lambda[0]\to\prn(N)$ in $\homo,$ with $N\in\add(M).$ Using that $H^0(f^\bullet): \Lambda\to M'$ is an $\add(M)$-preenvelope, there is $k: M'\to N$ in $\modu(\Lambda)$ such that $H^0(g^\bullet)=k\,H^0(f^\bullet).$ By the comparison lemma in homological algebra, there is a morphism $\overline{k^\bullet}: \pr(M')\to\pr(N)$ of cochain complexes such that $H^0(\overline{k^\bullet})=k.$ 
Note that $\overline{k^\bullet}$ can be restricted to a morphism $k^\bullet: \prn(M')\to\prn(N)$ of cochain complexes with $H^0(k^\bullet)=k.$ Then
$H^0(g^\bullet)=H^0(k^\bullet)H^0(f^\bullet),$ and thus by Lemma \ref{LemHsim}, we get that $g^\bullet=k^\bullet f^\bullet$ in $\homo.$ 
\

(b) Let $i\geq 0.$ Applying {\small $(-,\prn(M)[i]):=\HomK(-,\prn(M)[i])$} to $\eta,$ and using that $(\prn(M'),\prn(M)[i+1])=0$ (see Theorem \ref{NAIR, 3.4}), 
we get the following exact sequence\\
 {\small $(*)(\prn(M'),\prn(M)[i])\xrightarrow{(f^\bullet,\prn(M)[i])}(\Lambda[0],\prn(M)[i])\twoheadrightarrow(Q^\bullet,\prn(M)[i+1]).$ }
The item (a) and $i=0$ in $(*)$ give us that $\HomK(Q^\bullet,\prn(M)[1])=0.$
\

Let $i\geq 1.$ Then $H^i(\prn(M))=0$ and thus from Lemma $\ref{LemHsim},$ it follows that  $\HomK(\Lambda[0],\prn(M)[i])=0.$ Therefore, from $(*)$ we conclude that 
$\HomK(Q^\bullet,\prn(M)[i+1])=0.$ 
Hence $Q^\bullet\in{}^{\perp_{>0}}\prn(M).$
\end{proof}

\begin{lem}\label{l3NAIR, 2.2} Let $M\in\modu(\Lambda)$ be such that $M\in M^\perp.$ Then, for a family  $\{\eta_i: 0\to N_i\xrightarrow{f_i}M_i\to N_{i+1}\to0\}_{i=1}^n$ of exact sequences in $\modu(\Lambda),$  with  $\{M_i\}_{i=1}^n$ and 
$N_{n+1}$ in  $\add(M),$ the following statements hold true.
\begin{itemize}
\item[(a)] $N_i\in{}^\perp M$ $\forall$ $i\in[1,n].$
\item[(b)] $f_i:N_i\to M_i$ is an $\add(M)$-preenvelope $\forall$ $i\in[1,n].$\end{itemize}
\end{lem}
\begin{proof}
(a) It can be shown by reverse induction on $i\in[1,n].$

(b) It follows from (a), by applying the functor $\Hom_\Lambda(-,M'),$ with $M'\in\add(M),$ to each $\eta_i,$ and then using that $M\in M^\perp.$ 
\end{proof}

\begin{lem}\label{l4NAIR, 2.2} Let $n\in\mathbb{N}^+,$ $\pr\in\mathsf{C}^{[-n,0]}(\proj(\Lambda)),$ $Q^\bullet\in\Csf^{\leq 0}(\proj(\Lambda))$ such that $H^i(Q^\bullet)=0$ $\forall$ $i\in(-n,0),$ and let $f: H^0(\pr)\to H^0(Q^\bullet)$ be in $\modu(\Lambda).$ Then there is a morphism $f^\bullet: \pr\to Q^\bullet$ of cochain complexes such that $H^0(f^\bullet)=f.$
\end{lem}
\begin{proof} Let $p:=\Coker(d^{-1}_{P^\bullet}):P^0\to H^0(\pr)$ and $q:=\Coker(d^{-1}_{Q^\bullet}): Q^0\to H^0(Q^\bullet).$ Since $P^0\in\proj( \Lambda)$ and $q$ is an epimorphism, there exists $f^0: P^0\to Q^0$ such that $qf^0=fp.$ Using that $P^{-1}$ is projective, $\Ima(d^{-1}_{Q^\bullet})=\Ker(q)$ and 
$p\,d^{-1}_{P^\bullet}=0,$ we get the commutative diagram in $\modu(\Lambda)$
$$\xymatrix{P^{-1}\ar[r]^{d^{-1}_{P^\bullet}}\ar[d]_{f^{-1}}&P^0\ar[d]_{f^0}\ar[r]^p&H^0(\pr)\ar[r]\ar[d]^f&0\\
Q^{-1}\ar[r]_{d^{-1}_{Q^\bullet}}&Q^0\ar[r]_q&H^0(Q^\bullet)\ar[r]&0.}$$
Thus we have a morphism $\overline{f}^\bullet: \pr_{\geq-1}\to Q^\bullet_{\geq-1}$ of cochain complexes such that $H^0(\overline{f}^\bullet)=f.$ Since $H^i(Q^\bullet)=0$ $\forall$ $i\in(-n,0),$ by Lemma \ref{l1, NAIR, 3.4} (b), the morphism 
$\overline{f}^\bullet$ can be extended to a morphism $\overline{\overline{f}}^\bullet: \pr_{\geq-n}\to Q^\bullet_{\geq-n}$ of cochain complexes. Finally, since 
$P^i=0$ $\forall\, i<-n,$ the above morphism extends to a morphism $f^\bullet: \pr\to Q^\bullet$ of cochain complexes and $H^0(f^\bullet)=f.$
\end{proof}

\begin{lem}\label{l5NAIR, 2.2} For $M\in\modu(\Lambda)$ and 
$f^\bullet: \pr\to Q^\bullet$ in $\Ksf^{\leq 0}(\proj(\Lambda)),$ with $H^0(Q^\bullet)\in\add(M),$ the following statements are equivalent.
\begin{itemize}
\item[(a)] $H^0(f^\bullet):H^0(\pr)\to H^0(Q^\bullet)$ is an $\add(M)$-preenvelope.
\item[(b)] Any morphism $g^\bullet:P^\bullet\to P^\bullet(M')$ in $\Ksf(\proj(\Lambda)),$ 
with  $M'\in\add(M),$ 
factors through $f^\bullet: \pr\to Q^\bullet.$
\end{itemize}
\end{lem}
\begin{proof} (a) $\Rightarrow$ (b) Let $M'\in\add(M)$ and 
$g^\bullet: \pr\to\pr(M')$ be in $\Ksf(\proj(\Lambda)).$ Consider the morphism $g:=H^0(g^\bullet): H^0(\pr)\to M'.$ Since $H^0(f^\bullet)$ is an $\add(M)$-preenvelope, there is $h: H^0(Q^\bullet)\to M'$ such that $g=hH^0(f^\bullet).$ By the comparison lemma in homological algebra, there is a morphism $h^\bullet: Q^\bullet\to\pr(M')$ of cochain complexes such that $H^0(h^\bullet)=h$ and 
therefore $g^\bullet=h^\bullet f^\bullet$  in $\Ksf(\proj(\Lambda)).$ 
\

(b) $\Rightarrow$ (a) Let $M'\in\add(M)$ and $g: H^0(\pr)\to M'$ in 
$\modu(\Lambda).$ By the comparison lemma in homological algebra, there is a morphism $g^\bullet: \pr\to\pr(M')$ of cochain complexes such that 
$H^0(g^\bullet)=g.$ By hypothesis,  there exists a morphism 
$h^\bullet: Q^\bullet\to\pr(M')$ of cochain complexes such that 
$g^\bullet=h^\bullet f^\bullet$ in $\Ksf(\proj(\Lambda)).$ Therefore 
$g=H^0(h^\bullet)H^0(f^\bullet).$ 
\end{proof}

\begin{teo}\label{NAIR, 2.2} For $M\in\modu(\Lambda),$ the following statements hold true.
\begin{itemize}
\item[(a)] If $M$ is $\tau_{n,m}$-tilting in $\modu(\Lambda),$ then $M$ is $m$-tilting in $\modu(\Gamma_{\!\!M}).$
\item[(b)] Let $n\leq 2.$ If $M$ is $\tau_n$-rigid and sincere in $\modu(\Lambda)$ and $n$-tilting in $\modu(\Gamma_{\!\!M}),$ then $M$ is $\tau_{n,n}$-tilting in $\modu(\Lambda).$
\end{itemize} 
\end{teo}
\begin{proof} We consider the map $p:\Lambda\to \Gamma_{\!\!M},\;\lambda\mapsto \lambda+\ann(M).$ Note that, for $M'\in\add(M),$ we have that $\ann(M)\subseteq\Ker(\varphi)$ for any $\varphi\in\Hom_\Lambda(\Lambda, M').$
On the other hand, by Lemma \ref{lNAIR, 2.2} we have that $\resdim_{\add(\Lambda[0])}(\prn(M))\leq n$ in $\homo.$
\

(a) Let $M$ be $\tau_{n,m}$-tilting in $\modu(\Lambda).$ Since 
$M$ is faithful and $\tau_m$-rigid in $\modu(\Gamma_{\!\!M}),$ by  
Lemma \ref{l5. NAIR, 2.2}, we get that $\pd_{\Gamma_{\!\!M}}(M)\leq m$ and 
$M\in M^{\perp_{\Gamma_{\!\!M}}}.$ Thus, in order to prove that $M$ is $m$-tilting in $\modu(\Gamma_{\!\!M}),$ it is enough to show that $\coresdim_{\add(M)}(\Gamma_{\!\!M})<\infty$ in  $\modu(\Gamma_{\!\!M}).$
\

Since $\resdim_{\add(\Lambda[0])}(\prn(M))\leq n$ in $\homo$ and $\prn(M)$ is silting in $\homo,$ from \cite[Proposition 3.9 and Remark 3.3]{We13} we get that\\ 
$\coresdim_{\add(\prn(M))}(\Lambda[0])\leq n.$ Therefore, 
there is a family 
\begin{center}
$\{\eta_i: P^\bullet_i\overset{f_i^\bullet}{\to}\prn(M_i)\to P^\bullet_{i+1}\to P^\bullet_i[1]\}_{i=0}^{n-1},$
\end{center}
of distinguished triangles  in $\homo,$ with $P^\bullet_0:=\Lambda[0]$, $P^\bullet_n:=\prn(M_n)$ and $\{M_i\}_{i=0}^n\subseteq\add(M).$ Applying Lemma \ref{l6. NAIR, 2.2} (c) to the above family of distinguished triangles, we obtain the exact sequence $\Lambda\xrightarrow{H^0(f^\bullet_0)}M_0\to\cdots\to M_n\to0$ in $\modu(\Lambda).$ Then, by Lemma \ref{l2NAIR, 2.2} (b) and Proposition \ref{p1p} (a), we have that $H^0(f^\bullet_0): \Lambda\to M_0$ is an $\add(M)$-preenvelope. On the other hand, there is some $h: \Gamma_M\to M_0$ such that $hp=H^0(f^\bullet_0).$ Now, from the fact that $M$ is faithful in $\modu(\Gamma_{\!\!M}),$ there is a monomorphism $t:\Gamma_{\!\!M}\to M'$ in $\modu(\Lambda),$ with $M'\in\add(M).$ Hence, using that $H^0(f^\bullet_0)$ is an $\add(M)$-preenvelope, there exists $h':M_0\to M'$ such that $tp=h'H^0(f^\bullet_0)=h'hp,$ and thus 
$h'h=t$  since $p$ is an epimorphism. Now, $h: \Gamma_M\to M_0$ is a monomorphism since $t$ is so. Finally, using that $\Ima(h)=\Ima(H^0(f^\bullet_0)),$ we get the exact sequence $0\to\Gamma_M\xrightarrow{h}M_0\to\cdots\to M_{n-1}\to M_n\to 0$ in  $\modu(\Gamma_{\!\!M})$ and then 
$\mathrm{coresdim}_{\add(M)}(\Gamma_{\!\!M})\leq n$ in $\modu(\Gamma_{\!\!M}).$ 
\

(b) Let $M$ be $\tau_n$-rigid and sincere in $\modu(\Lambda)$ and $n$-tilting in $\modu(\Gamma_{\!\!M}).$ By Corollary \ref{CoroNAIR, 3.4} (b), we get that $M$ is $\tau_n$-rigid in $\modu(\Gamma_{\!\!M}).$ In order to show that $M$ is $\tau_n$-tilting in $\modu(\Lambda),$ by using Theorem \ref{NAIR, 3.4}, \cite[Theorem 3.5 and Remark 3.3]{We13} and $\resdim_{\add(\Lambda[0])}(\prn(M))\leq n$ in $\homo,$ it is enough to prove that $\coresdim_{\add(\prn(M))}(\Lambda[0])<\infty$ in $\homo.$
\

Since $\mathrm{coresdim}_{\add(M)}(\Gamma_{\!\!M})\leq n$ in $\modu(\Gamma_{\!\!M}),$ there is a family 
 $\{\mu_i:\;0\to N_i\overset{f_i}{\to}M_i\overset{g_i}{\to}N_{i+1}\to0\}_{i=0}^{n-1}$ of exact sequences in $\modu(\Gamma_{\!\!M}),$ with $N_0:=\Gamma_{\!\!M},$ 
$N_n:=M_n$ and $\{M_i\}_{i=0}^{n}$ in $\add(M).$ 

Consider $f_0p: \Lambda\to M_0.$ Since $\Lambda\in\proj(\Lambda)$ there is $f_0^\bullet: \Lambda[0]\to\prn(M_0)$ in $\homo$ with $H^0(f_0^\bullet)=f_0p.$ Moreover, this morphism can be completed to the distinguished triangle in $\homo$
$$\eta_0:\; \Lambda[0]\xrightarrow{f_0^\bullet}\prn(M_0)\to P_1^\bullet\to\Lambda[1],$$
where $P_1^\bullet:=\Cone(f_0^\bullet)\in\mathsf{C}^{[-n,0]}(\proj(\Lambda)).$ We assert that $\eta_0$ satisfies the following conditions:

\begin{enumerate}
\item[(I)] {\it $H^0(f_0^\bullet)=f_0p:\Lambda \to M_0$ is an $\add(M)$-preenvelope.}\\
 Indeed, let $h: \Lambda\to M'$ be in $\modu(\Lambda),$ with $M'\in\add(M).$ Then,  there is $\overline{h}: \Gamma_M\to M'$ such that $h=\overline{h}p.$ Now, by Lemma \ref{l3NAIR, 2.2} (b), we have that $f_0:\Gamma_{\!\!M}\to M_0$ is an $\add(M)$-preenvelope. Therefore, we get a morphism $h': M_0\to M'$ such that $\overline{h}=h'f_0.$ Hence $h=\overline{h}p=h'f_0p.$ 

\item[(II)] {\it $f_0^\bullet:\Lambda[0]\to \prn(M_0)$ is an $\add(\prn(M))$-preenvelope.}\\
Indeed, it follows from (I) and Proposition \ref{p1p} (a).

\item[(III)] $\pr_1\in{}^{\perp_{>0}}\prn(M).$\\ 
Indeed, it follows from (I) and Proposition \ref{p1p} (b).

\item[(IV)] $H^0(\pr_1)\simeq N_1.$\\
Indeed, by applying the cohomological functor $H^0$ to $\eta_0,$ we get the  exact sequence
$\Lambda\xrightarrow{f_0p} M_0\to H^0(\pr_1)\to 0.$ Thus 
$H^0(\pr_1)\simeq \Coker(f_0p)\simeq \Coker(f_0)\simeq N_1$ since $p$ is an epimorphism.
\end{enumerate}

Let $n=1.$ From (IV) and Proposition \ref{p1. NAIR, 2.2} there exists $Q\in\proj(\Lambda)$ such that $\pru(N_1)\coprod Q[1]=P_1^\bullet.$ Then, by (III), 
 Theorem \ref{NAIR, 3.4} and Lema \ref{lemtrunc} we have that 
 $\Hom_{\mathsf{K}(\proj(\Lambda))}(Q[1],\pr(M)[1])=0.$ Since $M$ is sincere, it follows from Lemma \ref{l3. NAIR, 2.2} that $Q=0.$ Thus $\pru(N_1)=P_1^\bullet,$ with $N_1=M_1\in\add(M),$ and then from $\eta_0$ we get $\coresdim_{\add(P^{\bullet}_{\geq -1}(M))}(\Lambda[0])\leq 1$ in $\homo.$
\

Let $n=2.$ Since $\pr_1\in\mathsf{C}^{[-2,0]}(\proj(\Lambda)),$ by (IV) and 
the comparison lemma in homological algebra, there is a morphism $f_1^\bullet:\pr_1\to\pr_{\geq-2}(M_1)$ in $\homo$ such that 
$H^0(f_1^\bullet)=f_1.$  Thus we have the distinguished triangle in $\homo$
$$\eta_1:\; \pr_1\xrightarrow{f_1^\bullet}\pr_{\geq-2}(M_1)\to\pr_2\to\pr_1[1],$$
where $\pr_2:=\Cone(f_1^\bullet)\in \mathsf{C}^{[-3,0]}(\proj(\Lambda)).$ Let us show that 
\begin{center}
$H^{-1}(\pr_2)=0$ and $H^0(\pr_2)\simeq N_2=M_2.$
\end{center}
Indeed, since $H^1(\pr_1)=0$ and $H^{-1}(\pr_{\geq-2}(M_1))=0,$ by applying the cohomological functor $H^0$ to $\eta_1,$ we obtain the exact sequence 
$0\to H^{-1}(\pr_2)\to N_1\xrightarrow{f_1}M_1\to H^0(\pr_2)\to0.$ Then $H^{-1}(\pr_2)=0$ ($f_1$ is a monomorphism) and $H^0(\pr_2)\simeq N_2.$

Using that $H^{-1}(\pr_2)=0$ and $H^0(\pr_2)\simeq N_2=M_2,$ we can construct a morphism $g^\bullet: \pr_{\geq-2}(M_2)\to\pr_2$ in $\homo$ such that $H^0(g^\bullet)=1_{M_2}.$ Hence we get the distinguished triangle in $\homo$
$$\varepsilon_1:\; \pr_{\geq-2}(M_2)\xrightarrow{g^\bullet}\pr_2\to T^\bullet\to\pr_{\geq-2}(M_2)[1],$$
where $T^\bullet:=\Cone(g^\bullet)\in\mathsf{C}^{[-3,0]}(\proj(\Lambda)).$ We assert that 
\begin{center}
$H^{-1}(T^\bullet)=0$ and $H^0(T^\bullet)=0.$
\end{center}
Indeed, by applying $H^0$ to $\varepsilon_1$ and using that $H^{-1}(\pr_2)=0$ and $H^1(\pr_{\geq-2}(M_2))=0,$ we get the exact sequence 
$0\to H^{-1}(T^\bullet)\to M_2\xrightarrow{1_{M_2}}M_2\to H^0(T^\bullet)\to 0.$ 
Thus our assertion follows.  Now, using that $T^\bullet\in\mathsf{C}^{[-3,0]}(\proj(\Lambda)),$ $H^{-1}(T^\bullet)=0$ and $H^0(T^\bullet)=0,$ we get from 
Lemma \ref{l1. NAIR, 2.2} that 
\begin{center}
$T^\bullet\simeq Q^\bullet$ in $\homo,$ where $Q^\bullet\in \mathsf{C}^{[-3,-2]}(\proj(\Lambda)).$ 
\end{center}

Now, let us show that $\pr_2\in{}^{\perp_{>0}}\pr(M).$ Indeed, consider the functor\\
$(-,\pr(M)):=\Hom_{\mathsf{K}(\proj(\Lambda))}(-,\pr(M)).$ Applying this functor to $\eta_1$ and since $(\pr_{\geq -2}(M_1),\pr(M)[1])=0$ (use Lemma \ref{lemtrunc} and Theorem \ref{NAIR, 3.4}), we get the following exact sequence
\begin{center}
$(\pr_{\geq-2}(M_1),\pr(M))\xrightarrow{(f_1^\bullet,\pr(M))}(\pr_1,\pr(M))\to(\pr_2,\pr(M)[1])\to 0.$
\end{center}
Moreover, Lemma \ref{l3NAIR, 2.2} (b) and Lemma \ref{l5NAIR, 2.2} imply that $(f_1^\bullet,\pr(M))$ is an epimorphism and thus $(\pr_2,\pr(M)[1])=0.$ Similarly, for $j\geq 1,$ from (III), Lemma \ref{lemtrunc} and Theorem \ref{NAIR, 3.4}, we get the exact sequence 
$$0=(\pr_1,\pr(M)[j])\to(\pr_2,\pr(M)[j+1])\to(\pr_{\geq-2}(M_1),\pr(M)[j+1])=0;$$
proving that $\pr_2\in{}^{\perp_{>0}}\pr(M).$
\

Let us show that $g^\bullet:\pr_{\geq -2}(M_2)\stackrel{\sim}{\to} \pr_2$ in $\homo.$ Indeed, for this, it is enough to show that $T^\bullet\simeq 0$ in $\homo.$ Let $j\geq 1,$  applying the functor $(-,\pr(M))$ to $\varepsilon_1$ and using that 
$\pr_2,\;\pr_{\geq-2}(M)\in{}^{\perp_{>0}}\pr(M),$ we obtain the  exact sequence 
$$0=(\pr_{\geq-2}(M),\pr(M)[j])\to(T^\bullet,\pr(M)[j+1])\to(\pr_2,\pr(M)[j+1])=0.$$ Thus $T^\bullet\in{}^{\perp_{\geq-2}}(\pr(M))$ and then $Q^\bullet[-2]\simeq T^\bullet[-2]\in{}^{\perp_{\geq0}}(\pr(M)).$ Then, from Lemma \ref{p3. NAIR, 2.2}, we get that $Q^\bullet[-2]\simeq 0$ in $\homo$ since $Q^\bullet[-2]\in \mathsf{C}^{[-1,0]}(\proj(\Lambda)).$ Hence $T^\bullet\simeq 0$ in $\homo.$
\

Finally, from the isomorphism $g^\bullet$ and the distinguished triangles $\eta_0$ and $\eta_1,$ we conclude that $\coresdim_{\add(P^{\bullet}_{\geq -2}(M))}(\Lambda[0])\leq 2$ in $\homo.$
\end{proof}

\begin{rk} The given proof of Theorem \ref{NAIR, 2.2} (b) does not work for $n\geq 3.$ As we will see in what follows, the problem is that we get a ``to big'' cochain complex with non zero cohomology in a relevant degree. Indeed, let $n=3.$ Then, doing as before, we can construct the same distinguished triangles $\eta_0$ and $\eta_1.$ In this case we have that $\pr_2\in \mathsf{C}^{[-4,0]}(\proj(\Lambda))$ since $\pr_1\in\mathsf{C}^{[-3,0]}(\proj(\Lambda)).$ Note that the length of $\pr_2$  could be bigger than the length of $\pr_{\geq 3}(M_2)$ and thus we can not lift 
$f_2:N_2\to M_2$ to a morphism $f^\bullet_2:\pr_2\to \pr_{\geq 3}(M_2)$ of cochain complexes. On the other hand, if we want to solve the above situation by lifting first the  identity map $1_{N_2}: N_2\to N_2,$ we can not do that because $\ann(M)\subseteq H^{-2}(\pr_2)$ and thus $H^{-2}(\pr_2)\neq 0$ if $\ann(M)\neq 0.$ 
\

Let us show that $\ann(M)\subseteq H^{-2}(\pr_2).$ By applying $H^0$ to $\eta_1$ and using that $H^i(\pr_{\geq 3}(M_1))=0,$ for $i=-1,-2,$ we get that $H^{-2}(\pr_2)\simeq H^{-1}(\pr_1).$ On the other hand, by applying $H^0$ to $\eta_0$ and using that $H^{-1}(\pr_{\geq 3}(M_0))=0,$ it follows that $\ann(M)\subseteq\Ker(f_0p)\simeq H^{-1}(\pr_1).$ 
\end{rk}

\begin{cor}\label{c, NAIR, 2.2} Let $M\in\modu(\Lambda)$ and $n=1,2.$ Then the following statements are equivalent.
\begin{itemize}
\item[(a)] $M$ is $\tau_{n,n}$-tilting in $\modu(\Lambda).$
\item[(b)] $M$ is sincere and $\tau_n$-rigid in $\modu(\Lambda)$ and $M$ is $n$-tilting in $\modu(\Gamma_{\!\!M}).$
\end{itemize}
\end{cor} 
\begin{proof} It follows from Theorem \ref{NAIR, 2.2} and Theorem \ref{c1} (a).
\end{proof}
  
  The next example shows a $\tau_{2,2}$-tilting $\Lambda$-module of infinite projective dimension which is not $\tau$-rigid.  
    
\begin{ex}\label{ExImp}Let $k$ be a field and let $\Lambda$ be the quotient path $k$-algebra $kQ/R^2,$ where $Q$ is the quiver 
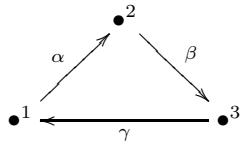
\begin{center}
$\xymatrix{&&\bullet^2\ar[dr]^{\beta}\\
&\bullet^1\ar[ru]^{\alpha}&&\bullet^3\ar[ll]^{\gamma}}$
\end{center}
and $R$ is the ideal in $kQ$ generated by all the arrows in $Q.$ The Auslander-Reiten quiver of $\modu(\Lambda)$ is
{\small 
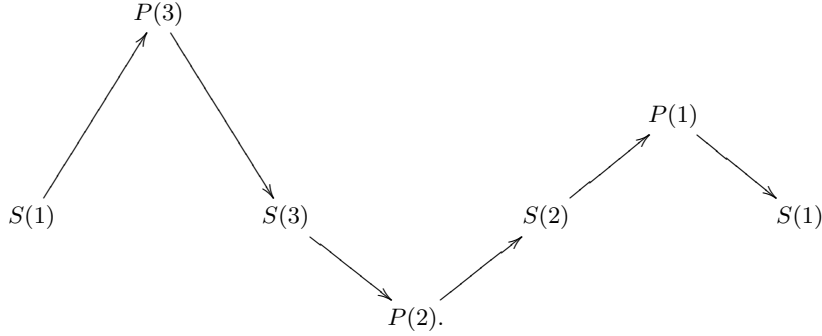
$$\xymatrix{&P(3)\ar[ddr]\\
&&&&&P(1)\ar[rd]\\
S(1)\ar[ruu]&&S(3)\ar[rd]&&S(2)\ar[ru]&&S(1)\\
&&&P(2).\ar[ru]}$$}
Let $M:=P(2)\oplus P(1)\oplus S(1).$ Note that $M$ is sincere,  
$M^{\perp_{\tau_2}}={}^{\perp_0}S(3)\cap S(1)^{\perp_1}=\add(M)$ and thus $M$ is $\tau_2$-rigid in $\modu(\Lambda).$ It can be shown that $\Gamma_{\!\!M}:=\Lambda/\ann(M)=kQ'/R'^2,$ where $Q'$ is the quiver $\xymatrix{\bullet^1\ar[r]^\delta&\bullet^2\ar[r]^\iota&\bullet^3}$ and $R'=<\iota\delta>.$ Moreover, the Auslander-Reiten of $\modu(\Gamma_{\!\!M})$ is
{\small 
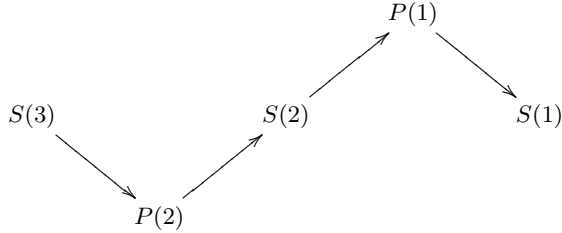
$$\xymatrix{&&&P(1)\ar[rd]\\
S(3)\ar[rd]&&S(2)\ar[ru]&&S(1)\\
&P(2)\ar[ru]}$$}
and thus $M$ is $2$-tilting in $\modu(\Gamma_{\!\!M}).$ Therefore, from Corollary \ref{c, NAIR, 2.2},  $M$ is $\tau_{2,2}$-tilting in $\modu(\Lambda).$
Note that $\Omega^3(S(1))=S(1)$ and hence $\pd(S(1))$ is infinite. Finally, 
${}^{\perp_0}\tau M={}^{\perp_0}S(2)=\add(S(3)\oplus P(3)\oplus S(1)\oplus P(1))$ in $\modu(\Lambda)$ and then $M$ is not $\tau$-rigid.
\end{ex}

\begin{lem}\label{l1,c} Let $J$ be a nilpotent ideal of $\Lambda.$  Then 
$\findim(\Lambda)\leq \pd_\Lambda(\Lambda/J)+\findim(\Lambda/J).$
\end{lem}
\begin{proof} We can  adapt the proof of \cite[Proposition 2.3]{Su21} to our situation as follows. We can assume that $n:=\pd_\Lambda(\Lambda/J)$ is finite (otherwise there is nothing to prove). Let $M\in\modu(\Lambda)$ with $\pd_\Lambda(M)$ finite. In particular, for $M_n:=\Omega^n(M),$ we have that $\pd_\Lambda(M_n)$ is also finite. Then, we can continue the proof as in \cite[Proposition 2.3]{Su21} and then 
$\pd_{\Lambda/J}(M_n/JM_n)=\pd_\Lambda(M_n)<\infty.$ Therefore
$$\pd_\Lambda(M)\leq n+\pd_\Lambda(M_n)\leq n+\findim(\Lambda/J),$$
proving the result.
\end{proof}

\begin{lem}\label{LbAnn} Let $M\in\modu(\Lambda)$ be sincere. Then, $\ann(M)\subseteq\rad(\Lambda)$ and $rk(\Lambda)=rk(\Gamma_{\!\!M}).$
\end{lem}
\begin{proof} Since $M$ is sincere, there is no idempotent $0\neq e\in\Lambda$ such that $e\in\ann(M).$ Therefore,  $\ann(M)\cap P\subseteq\rad(P),$ for any indecomposable projective $\Lambda$-module. Hence  $\ann(M)\subseteq \rad(\Lambda).$
\

 In order to prove that $rk(\Lambda)=rk(\Gamma_{\!\!M}),$ it is enough to show that 
 $rk\,K_0(\Lambda)=rk\,K_0(\Gamma_{\!\!M}).$ Since $\modu(\Gamma_{\!\!M})$ is closed under submodules in $\modu(\Lambda),$ we get that any simple $\Gamma_{\!\!M}$-module is a simple $\Lambda$-module and thus $rk(\Gamma_{\!\!M})\leq rk(\Lambda).$ Finally, using now that $\modu(\Gamma_{\!\!M})$ is both closed under submodules and quotients in $\modu(\Lambda),$ we get that any simple $\Lambda$-module appearing as composition factor of $M$ is a simple $\Gamma_{\!\!M}$-module. Then 
 $rk(\Lambda)\leq rk(\Gamma_{\!\!M})$ since $M$ is sincere.
\end{proof}

The following result is a generalization of \cite[Theorem 2.5]{Su21}, and to obtain it, we take $n=m=1$ and $\Lambda$ of finite global dimension in Corollary \ref{Corfindim}.

\begin{cor}\label{Corfindim} Let $M$ be $\tau_{n,m}$-tilting in $\modu(\Lambda)$  and $\Gamma_{\!\!M}:=\Lambda/\ann(M).$ Then 
$\findim(\Lambda)\leq \findim(\End_{\Gamma_{\!\!M}}(M)^{op})+\pd_\Lambda(M)+m.$
\end{cor}
\begin{proof}  By Theorem \ref{c1} (a) and Lemma \ref{LbAnn}, we get that  $\ann(M)$ is nilpotent. Then we can apply Lemma \ref{l1,c} to $J:=\ann(M)$ and so we get that $\findim(\Lambda)\leq \pd_\Lambda(\Gamma_{\!\!M})+\findim(\Gamma_{\!\!M}).$
On the other hand, by Theorem \ref{NAIR, 2.2} we know that $M$ is $m$-tilting in $\modu(\Gamma_{\!\!M}),$ which implies that $\Gamma_{\!\!M}\in (\add(M))^\vee$ in 
$\modu(\Lambda).$ Thus by \cite[Lemma 2.13 (a)]{MS06} we get that    $\pd_\Lambda(\Gamma_{\!\!M})\leq \pd_\Lambda((\add(M))^\vee)=\pd_\Lambda(M).$ Finally, from \cite[Theorem 1.1]{PX09}, it follows that 
$\findim(\Gamma_{\!\!M})\leq \findim(\End_{\Gamma_{\!\!M}}(M)^{op})+m.$
 \end{proof}
 
\begin{rk}\label{Funfin=} Let $J\unlhd \Lambda$ and $\overline{\Lambda}:=\Lambda/J.$ Note that 
$\modu(\overline{\Lambda})$ is closed under submodules in $\modu(\Lambda)$ and it is finitely cogenerated since $\overline{\Lambda}$ in an Artin algebra. Then, by \cite[Proposition 4.7 (c)]{AS80}, we get that $\modu(\overline{\Lambda})$ is funtorially finite in $\modu(\Lambda).$ Thus, for any 
$\X\subseteq\modu(\overline{\Lambda}),$ we have that $\X$ is preenveloping (precovering) in $\modu(\Lambda)$ if and only if it is so in $\modu(\overline{\Lambda}).$ 
\end{rk} 

 We want to give a characterization to know when a $\tau_n$-tilting $M\in\modu(\Lambda)$ is $\tau_m$-rigid in $\modu(\Gamma_{\!\!M}).$  In order to do that, we start with the following notions and preliminary results.

Let $\omega\subseteq\modu(\Lambda),$ $\ann(\omega):=\cap_{W\in\omega}\ann(W)$ and 
$\Gamma_{\!\omega}:=\Lambda/\ann(\omega).$ We say that the class $\omega$ is sincere if it contains a sincere $\Lambda$-module. Note that if $\omega=\add(W),$ then 
$\ann(\omega)=\ann(W),$ $\omega^{\perp_{\tau_n}}=W^{\perp_{\tau_n}}$ and $\omega^\perp=W^\perp.$

\begin{lem}\label{rtomega} For $\omega\subseteq\modu(\Lambda),$ the following statements hold true.
\begin{itemize}
\item[(a)] The class $\omega^{\perp_{\tau_n}}$ is right thick in $\modu(\Lambda).$
\item[(b)] If $\inj(\Gamma_{\!\omega})\subseteq \omega^{\perp_{\tau_n}}\subseteq\modu(\Gamma_{\!\omega}),$ then $(\omega^{\perp_{\tau_n}})^{\vee}_n=\modu(\Gamma_{\!\omega}).$
\end{itemize}
\end{lem}
\begin{proof} (a) Since 
$\omega^{\perp_{\tau_n}}=\cap_{W\in\omega}\,W^{\perp_{\tau_n}},$ the result follows from Corollary \ref{CoroNAIR, 3.4} (d).
\

(b) Assume that $\inj(\Gamma_{\!\omega})\subseteq \omega^{\perp_{\tau_n}}\subseteq\modu(\Gamma_{\!\omega}).$ Let $N\in \modu(\Gamma_{\!\omega}).$ Consider the family 
$\{\eta_i:\;0\to N_i\to I_i\to N_{i+1}\}_{i=0}^{n-1}$ of exact sequences in $\modu(\Gamma_{\!\omega}),$ where $\{I_k\}_{k=0}^{n-1}\subseteq \inj(\Gamma_{\!\omega})$ and $N_0:=N.$ It is enough to show that $N_n\in W^{\perp_{\tau_n}},$ for any $W\in \omega.$
\

Let $W\in\omega.$ By Corollary \ref{CoroNAIR, 3.4} (a),  
${}^{\perp_0}\tau_n(W)\subseteq W^{\perp_n}$ and $W^{\perp_{\tau_n}}\subseteq\cap_{i=1}^n W^{\perp_i}.$ Moreover $\inj(\Gamma_{\!\omega})\subseteq W^{\perp_{\tau_n}}.$ Let us show that the following (i) and (ii) hold true.

(i) $N_{i+1}\in {}^{\perp_0}\tau_n(W)$ for any $i\in[0,n-1].$\\
Indeed, it follows by applying $\Hom_\Lambda(-,\tau_n(W))$ to $\eta_i$ and using that $\inj(\Gamma_{\!\omega})\subseteq {}^{\perp_0}\tau_n(W).$
\

(ii) $N_n\in \cap_{i=1}^{n-1} W^{\perp_i}.$\\
Indeed, let $i\in[1,n-1].$ Applying the functor $\Hom_\Lambda(W,-)$ to $\eta_n$ and 
since $\inj(\Gamma_{\!\omega})\subseteq \cap_{i=1}^n W^{\perp_i},$ we get that 
$\Ext^i_\Lambda(W,N_n)\simeq \Ext^{i+1}_\Lambda(W,N_{n-1}).$ Thus, by repeating the same argument for $\eta_{n-1}, \cdots, \eta_i,$ we get the chain of isomorphisms
$$\Ext^i_\Lambda(W,N_n)\simeq \Ext^{i+1}_\Lambda(W,N_{n-1})\simeq\cdots\simeq 
\Ext^n_\Lambda(W,N_i).$$
Then by (i) $\Ext^n_\Lambda(W,N_i)=0$ and thus (ii) follows. Finally, from (i) and (ii), we conclude that $N_n\in W^{\perp_{\tau_n}}.$
\end{proof}

\begin{defi}\label{compClass} A class $\omega\subseteq\modu(\Lambda)$ is called $\tau_{n,m}$-compatible if the following conditions hold true for $\omega.$
\begin{itemize}
\item[(C1)] $\add(\omega)=\omega$ is sincere and a relative generator in $\omega^{\perp_{\tau_n}}.$
\item[(C2)] $\inj(\Gamma_{\!\omega})\subseteq\omega^{\perp_{\tau_n}}\subseteq\modu(\Gamma_{\!\omega}).$ 
\item[(C3)] $\omega^{\perp_{\tau_n}}$ is  preenveloping in $\modu(\Gamma_{\!\omega}).$
\item[(C4)] $\pd_{\Gamma_{\!\omega}}(\omega)\leq m.$
\end{itemize}
We denote by $\mathfrak{C}_{\tau_{n,m}}(\Lambda)$ the class of all the $\omega\subseteq\modu(\Lambda)$ which are $\tau_{n,m}$-compatible.
\end{defi}

\begin{teo}\label{Comp-tau} For a $\tau_{n,m}$-compatible class $\omega\subseteq\modu(\Lambda),$ the following statements hold true.
\begin{itemize}
\item[(a)] There is a basic, sincere and $\tau_n$-rigid $T\in\modu(\Lambda).$
\item[(b)] $T$ is $\min\{m,n\}$-tilting in $\modu(\Gamma_{\!T})$ and $\omega=\add(T).$
\item[(c)] $\omega^{\perp_{\tau_n}}=T^{\perp_{\tau_n}}=T^{\perp_{\Gamma_{\!T}}}$ and  
$rk(T)=rk(\Lambda).$
\end{itemize}
\end{teo}
\begin{proof} By Lemma \ref{rtomega} (a), we have that $\omega^{\perp_{\tau_n}}$ is right thick in $\modu(\Lambda)$ and thus it is so in $\modu(\Gamma_{\!\omega}).$ Furthermore, from Lemma \ref{rtomega} (b), we know that $(\omega^{\perp_{\tau_n}})^{\vee}_n=\modu(\Gamma_{\!\omega}).$
Then, by the above properties of $\omega^{\perp_{\tau_n}}$ in $\modu(\Gamma_{\!\omega}),$ we get from  Theorem \ref{AR91, 5.5} that there is a tilting $T$ in $\modu(\Gamma_{\!\omega})$ such that $\omega^{\perp_{\tau_n}}=T^{\perp_{\Gamma_{\!\omega}}}.$
\

 Let us prove that 
 $\omega=\add(T).$ Indeed, since $T\in \omega^{\perp_{\tau_n}}$ and $\omega$ is a relative generator in $\omega^{\perp_{\tau_n}},$  there is an exact sequence $0\to K\to W\to T\to0$ in 
 $\modu(\Gamma_{\!\omega}),$ where $K\in \omega^{\perp_{\tau_n}}$ and $W\in\omega.$ 
 Thus the above sequence splits and then $\add(T)\subseteq \omega.$ Let now $W\in\omega\subseteq \omega^{\perp_{\tau_n}}=T^{\perp_{\Gamma_{\!\omega}}}.$ Hence, by Theorem \ref{AR91, 5.5}, the is an exact sequence $0\to X\to T'\to W\to 0,$ with $T'\in\add(T)$ and $X\in T^{\perp_{\Gamma_{\!\omega}}}=\omega^{\perp_{\tau_n}}\subseteq\cap_{i=1}^n\omega^{\perp_i}$ (see Corollary \ref{CoroNAIR, 3.4} (a)). Thus, the preceding exact sequence splits and then $\omega\subseteq\add(T).$ Therefore,  $\omega=\add(T)$ and then $T$ is sincere in $\modu(\Lambda)$ since $\omega$ has a sincere $\Lambda$-module. In particular, we also have that 
 $\Gamma_{\!\omega}=\Gamma_{\!T}$ and $\omega^{\perp_{\tau_n}}=T^{\perp_{\tau_n}}.$ On the other hand, from Lemma \ref{LbAnn} and \cite[Theorem 1.19]{Mi86}, we get that  $rk(\Lambda)=rk(\Gamma_{\!T})=rk(T).$ Moreover $\pd_{\Gamma_{\!\omega}}(T)=\pd_{\Gamma_{\!\omega}}(\omega)\leq m$ and thus $T$ is $m$-tilting in $\modu(\Gamma_{\!\omega}).$  Note that $T$ is $\tau_n$-rigid in $\modu(\Lambda)$ since $\add(T)=\omega\subseteq \omega^{\perp_{\tau_n}}.$ Finally, by 
 \cite[Corollary 5.12 (e1)]{AM21}, we get that $\pd_{\Gamma_{\!T}}(T)\leq n$ since 
 $(\omega^{\perp_{\tau_n}})^{\vee}_n=\modu(\Gamma_{\!\omega}).$
 \end{proof}

\begin{teo}\label{EquiTaurig} Let $M$ be $\tau_n$-tilting in $\modu(\Lambda)$ and $n\geq m\geq 1.$ Then, the following statements are equivalent.
\begin{itemize}
\item[(a)] $M$ is $\tau_m$-rigid in $\modu(\Gamma_{\!\!M}).$
\item[(b)] $\add(M)\in\mathfrak{C}_{\tau_{n,m}}(\Lambda).$
\item[(c)] $\pd_{\Gamma_{\!\!M}}(M)\leq m,$ $M^{\perp_{\tau_n}}$ is preenveloping in $\modu(\Lambda)$ and 
 $\inj(\Gamma_{\!\!M})\subseteq M^{\perp_{\tau_n}}.$
\item[(d)] $M$ is $m$-tilting in $\modu(\Gamma_{\!\!M}).$
\item[(e)] $M$ is partial $m$-tilting in $\modu(\Gamma_{\!\!M}).$
\end{itemize}
\end{teo}

\begin{proof} Since $M$ is $\tau_n$-tilting, by Proposition \ref{p4. NAIR, 2.2} (b) we have that $M^{\perp_{\tau_n}}=\gen_n(M)=\gen_n({}_{\Gamma_{\!\!M}}M)\subseteq\modu(\Gamma_{\!\!M}).$ Moreover, from Theorem \ref{c1} (a), $M$ is sincere and $rk(M)=rk(\Lambda).$ Finally, by Proposition \ref{l1, NHa95}, $\add(M)$ is a relative generator 
in $M^{\perp_{\tau_n}}.$

 (a) $\Rightarrow$ (b) By Theorem \ref{NAIR, 2.2} $M$ is $m$-tilting in $\modu(\Gamma_M).$ Hence by Theorem \ref{Ba04, 3.11} 
 $\gen_n({}_{\Gamma_{\!\!M}}M)=\gen_m({}_{\Gamma_{\!\!M}}M)=M^{\perp_{\Gamma_{\!\!M}}}$ and thus 
 $\inj(\Gamma_{\!\!M})\subseteq M^{\perp_{\Gamma_{\!\!M}}}=M^{\perp_{\tau_n}}.$
 \
 
 We prove now that $M^{\perp_{\tau_n}}$ is preenveloping. 
 Indeed, since $M$ is $m$-tilting in $\modu(\Gamma_{\!\!M}),$ by  Theorem \ref{AR91, 5.5} we get that $M^{\perp_{\Gamma_{\!\!M}}}=M^{\perp_{\tau_n}}$ is preenveloping in 
 $\modu(\Gamma_{\!\!M}).$  
\ 

(b) $\Leftrightarrow$ (c)  It follows from by Remark \ref{Funfin=}.
\

(b) $\Rightarrow$ (d) It follows from Theorem \ref{Comp-tau}.
\ 
 
(d) $\Rightarrow$ (e)  It is clear. 
 \
 
(e) $\Rightarrow$ (a) It follows from Corollary \ref{CoroNAIR, 3.4} (b).
\end{proof}

\section{Some open questions}

In this section we will address some questions that have left this research. 
\vspace{0.2cm}

{\bf Conjecture 1:} {\it If $M$ is $\tau_n$-tilting in $\modu(\Lambda),$ then $M^{\perp_{\tau_n}}$ is preenveloping.}
\vspace{0.2cm}

Let us show some cases where Conjecture 1 is true: (1) $M$ is $n$-tilting (see Theorem \ref{AR91, 5.5} (a) and Corollary \ref{CoroNAIR, 3.4} (b)), (2)  $M$ is $\tau$-tilting (see \cite[Theorem 2.10]{AIR14}) and (3) $M$ is  $\tau_{n,m}$-tilting for $n\geq m\geq 1$ (see Theorem \ref{EquiTaurig}).

From the view point of approximation theory, is it possible to have some kind of left approximations where the class $M^{\perp_{\tau_n}}$  is involved?. If we consider ``relative'' left approximations, in the sense that we introduce in what follows, we will give a positive answer to Conjecture 1 but just for the relative context.

\begin{defi} Let $\X,$ $\Y$ be classes of objects in $\modu(\Lambda).$ We say that 
a morphism $f:C\to Y$ in $\modu(\Lambda)$ is a $(\X,\Y)$-preenvelope of $C$ if $Y\in\Y$ and for any $X\in\X$ the map $\Hom_\Lambda(f,X):\Hom_\Lambda(Y,X)\to \Hom_\Lambda(C,X)$ is surjective. In case any $C\in\modu(\Lambda)$ has an $(\X,\Y)$-preenvelope, we say that the pair $(\X,\Y)$ is preenveloping.
\end{defi}

\begin{pro} The pair 
$(M^{\perp_{\tau_n}}, H^0(\prn(M)^{\perp_{>0}}))$ is preenveloping, for any  $\tau_n$-tilting $M$ in $\modu(\Lambda).$ 
\end{pro}
\begin{proof} Let $M$ be $\tau_n$-tilting in $\modu(\Lambda).$ Consider $N\in\modu(\Lambda).$ Then, by \cite[Proposition 2.23 (c)]{AI12} there exist a distinguished triangle in $\homo$
$$\eta:\; Q^\bullet \to \prn(N)\xrightarrow{g^\bullet}P^\bullet\to Q^\bullet[1],$$
where $P^\bullet\in\prn(M)^{\perp_{>0}}$ and $Q^\bullet\in {}^{\perp_0}(\prn(M)^{\perp_{>0}}).$ We assert that $H^0(g^\bullet):N\to H^0(P^\bullet)$ is a 
$(M^{\perp_{\tau_n}},H^0(\prn(M)^{\perp_{>0}}))$-preenvelope of $N.$ Indeed, let 
$f:N\to K$ be in $\modu(\Lambda),$ with $K\in M^{\perp_{\tau_n}}.$ In particular, by Theorem \ref{NAIR, 3.4} $\prn(K)\in \prn(M)^{\perp_{>0}}$ and thus 
$\Hom_{\homo}(Q^\bullet,\prn(K))=0.$
\

By the comparison lemma in homological algebra, there exists $f^\bullet:\prn(N)\to \prn(K)$ in $\homo$ with $H^0(f^\bullet)=f.$ Applying the functor $(-,\prn(K)):=\Hom_{\homo}(-,\prn(K))$ to $\eta$ and since $(Q^\bullet,\prn(K))=0,$ we get that
 $$(g^\bullet, \prn(K)):(P^\bullet,  \prn(K))\to ( \prn(N), \prn(K))$$ is surjective. Hence there exists $h^\bullet:P^\bullet\to  \prn(K)$ in $\homo$ such that 
$h^\bullet g^\bullet= f^\bullet.$ Then $f=H^0(h^\bullet) H^0(g^\bullet)$ and thus the result follows.
\end{proof}

\begin{rk} Let $M$ be $\tau_n$-rigid in $\modu(\Lambda).$ Then, by Theorem \ref{NAIR, 3.4}, we get that 
$M^{\perp_{\tau_n}}\subseteq H^0(\prn(M)^{\perp_{>0}}).$ Thus, if $M^{\perp_{\tau_n}}= H^0(\prn(M)^{\perp_{>0}}),$ then Conjecture 1 would have a positive answer in this case.
\end{rk}

In the following example we show that we could have a $\tau_{n}$-tilting $M$ such  that $M^{\perp_{\tau_n}}$ is preenveloping and however 
$M^{\perp_{\tau_n}}\neq H^0(\prn(M)^{\perp_{>0}}).$ 

\begin{ex} Let $\Lambda$ be the algebra given in Example \ref{Ejp1}. Consider the $\Lambda$-module
$M:=I(1)\oplus I(2)\oplus I(3).$ In particular, it is clear that $M\in M^\perp.$
\

Note that $\coresdim_{\add(M)}(\Lambda)\leq 4$ since we have the following exact sequences
$$0\to P(1)\to I(3)\to I(2)\to I(1)\to I(3)\to I(2)\to0,$$
$$0\to P(2)\to I(1)\oplus I(3)\to I(2)\oplus I(2)\to I(1)\to I(3)\to I(2)\to0,$$
$$0\to P(3)\to I(1)\to I(2)\to I(1)\to I(3)\to I(2)\to0.$$
On the other hand, using the following minimal projective resolutions
$$0\to P(3)\to P(2)\to P(1)\to P(3)^2\to P(2)\oplus P(1)\to I(3)\to0,$$
$$0\to P(3)\to P(2)\to P(1)\to P(3)\to P(1)\to I(2)\to0,$$
$$0\to P(3)\to P(2)\to P(1)\to P(3)\to P(2)\to I(1)\to0,$$
we get that $\pd(M)\leq 4$ and thus $M$ is $4$-tilting (and then $\tau_4$-tilting) in $\modu(\Lambda).$ By Corollary \ref{CoroNAIR, 3.4} (b), $M^{\perp_{\tau_4}}=M^\perp$  which is preenveloping by Theorem \ref{AR91, 5.5} (a).

Let $f: I(3)\to I(1)$ be the morphism in $\modu(\Lambda)$ such that  $\Ima(f)=S(1).$ By Lemma \ref{l1, NAIR, 3.4} (b), there exists $f^\bullet: \pr_{\geq-4}(I(3))\to\pr_{\geq-4}(I(1))$ in $\homo$ with $H^0(f^\bullet)=f.$ Thus we have the distinguished triangle in $\homo$
$$\eta:\;\pr_{\geq-4}(I(3))\xrightarrow{f^\bullet}\pr_{\geq-4}(I(1))\to Q^\bullet\to\pr_{\geq-4}(I(3))[1],$$ 
where $Q^\bullet:=\Cone(f^\bullet).$ Since $\pr_{\geq-4}(I(3)),$ $\pr_{\geq-4}(I(1))\in\pr_{\geq-4}(M)^{\perp_{>0}}$ (see Theorem \ref{NAIR, 3.4}), it is clear that 
$Q^\bullet\in\pr_{\geq-4}(M)^{\perp_{>0}}.$ Applying the cohomological functor $H^0$ to $\eta,$  we have the exact sequence $I(3)\to I(1)\to H^0(Q^\bullet)\to 0$ and thus 
$H^0(Q^\bullet)\simeq \Coker(f)=\begin{matrix}2\\3\end{matrix}.$ Consider the following exact diagram in $\modu(\Lambda)$
$$\xymatrix{0\ar[r]&S(3)\ar[r]^j\ar[d]_i&P(1)\ar[r]&I(2)\ar[r]&0\\
&H^0(Q^\bullet),}$$
where $i$ and $j$ are the inclusion maps. Since $\mathrm{Hom}_\Lambda(P(1),H^0(Q^\bullet))=0,$ it follows that $i$ does not factor through $j$. Therefore $\Ext^1_\Lambda(I(2),H^0(Q^\bullet))\neq0,$ proving that $H^0(Q^\bullet)\notin M^{\perp_{\tau_4}}.$ 
\end{ex}

{\bf Conjecture 2:} {\it If $M$ is $\tau_n$-tilting in $\modu(\Lambda),$ then $M$ is $\tau_{n,n}$-tilting.}
\vspace{0.2cm}

Note that, from Theorem \ref{EquiTaurig}, we have that: if Conjecture 2 is true then Conjecture 1 is also true. 

\begin{pro}\label{n=1} If $n=1$ then Conjecture 2 is true.
\end{pro}
\begin{proof} Let $M$ be $\tau_1$-tilting in $\modu(\Lambda).$ By \cite[Lemma VIII.5.2]{ASS06},
$\tau_{\Gamma_{\!\!M}}(M)$  is a $\Lambda$-submodule of $\tau_\Lambda(M)$  and thus 
$\Hom_{\Gamma_{\!\!M}}(M,\tau_{\Gamma_{\!\!M}}(M))\subseteq\Hom_\Lambda(M,\tau_\Lambda(M))=0.$ Therefore $M$ is $\tau_1$-tilting in $\modu(\Gamma_{\!\!M}).$
\end{proof}

{\bf Conjecture 3:} {\it If $M$ is $\tau_n$-rigid and sincere in  $\modu(\Lambda)$ and $m$-tilting in $\modu(\Gamma_{\!\!M}),$ with $n\geq m\geq 1,$ then $M$ is 
$\tau_n$-tilting in $\modu(\Lambda).$}
\vspace{0.2cm}

 The relevance of Conjecture 3 is that the validity of such conjecture implies the so called ``$\tau_{n,m}$-tilting correspondence theorem''  which is stated as follows:

\begin{teo}\label{corresT}  If Conjecture 3 is true, then for $M\in\modu(\Lambda)$ and $n\geq m\geq 1,$  the following statements (a), (b) and (c) are equivalent.
\begin{itemize}
\item[(a)] $M$ is $\tau_{n,m}$-tilting in $\modu(\Lambda).$
\item[(b)] $M$ is $\tau_n$-rigid and sincere in  $\modu(\Lambda)$ and $m$-tilting in $\modu(\Gamma_{\!\!M}).$
\item[(c)] $\add(M)\in\mathfrak{C}_{\tau_{n,m}}(\Lambda).$
\end{itemize}
Moreover, the map $M\mapsto\add(M)$ induces a bijection between iso-classes of 
basic $\tau_{n,m}$-tilting $\Lambda$-modules and elements in the class $\mathfrak{C}_{\tau_{n,m}}(\Lambda),$ whose inverse is the map $\omega\mapsto \oplus_{W\in\mathrm{ind}(\omega)}\,W.$
\end{teo} 
\begin{proof} It follows from Theorem \ref{Comp-tau} and Theorem \ref{EquiTaurig}, and we left the details to the reader.
\end{proof}

Notice that Conjecture 3 is true for $n=m=1,2$ (see Theorem \ref{NAIR, 2.2} (b)). Thus, until now, we have proved such ``$\tau_{n,m}$-tilting correspondence theorem'' only for $n=m=1,2.$

\footnotesize

\vskip3mm \noindent Luis Armando Mart\'inez Gonz\'alez\\ Instituto de Matem\'aticas\\ Universidad Nacional Aut\'onoma de M\'exico.\\ 
Circuito Exterior, Ciudad Universitaria\\
C.P. 04510, M\'exico, D.F. MEXICO.\\ {\tt luis.martinez@matem.unam.mx}

\vskip3mm \noindent Octavio Mendoza Hern\'andez\\ Instituto de Matem\'aticas\\ Universidad Nacional Aut\'onoma de M\'exico.\\ 
Circuito Exterior, Ciudad Universitaria\\
C.P. 04510, M\'exico, D.F. MEXICO.\\ {\tt omendoza@matem.unam.mx}


\begin{thebibliography}{22}
\bibitem[AIR14]{AIR14} T. Adachi, O. Iyama and I. Reiten. \emph{$\tau$-tilting theory}. Composotio Math. 150 (2014), 415-452. http://doi.org/10.1112/S0010437X13007422. 

\bibitem[AI12]{AI12} T. Aihara and O. Iyama. \emph{Silting mutation in triangulated categories}. J. Lond. Math. Soc. 85 (2012), 633-668.  https://doi.org/10.1112/jlms/jdr055.

\bibitem[AH19]{AH19} L. Angeleri H\"ugel. \emph{Silting Objects}. Bull. London Math. Soc. 51 (2019), 658-690.  https://doi.org/10.1112/blms.12264.

\bibitem[AM21]{AM21} A. Argudin, O. Mendoza. \emph{ Relative tilting theory in abelian categories II: $n$-$\X$-tilting theory.}  arXiv:2112.14873v1 [math.RT] 30 Dec 2021.

\bibitem[ASS06]{ASS06} I. Assem, D. Simson and A. Skowronski. \emph{Elements of the representation theory of associative algebras}. Vol. 65 (Cambridge University Press, Cambridge, 2006). https://doi.org/10.1017/CBO9780511614309.

\bibitem[AB89]{AB89} M. Auslander and R.-O. Buchweitz. \emph{The homological theory of maximal Cohen-Macaulay approximations}. M\'em. Soc. Math. France (N.S.), (38):5–37 (1989). Colloque en l’honneur de Pierre Samuel (Orsay, 1987). https://doi.org/10.24033/msmf.339.

\bibitem[AR75]{AR75} M. Auslander and I. Reiten. \emph{Representation theory of Artin algebras III: almost split sequences}. Comm. Algebra 3 (1975), 239-294. https://doi.org/10.1080/00927877508822046.

\bibitem[AR91]{AR91} M. Auslander and I. Reiten. \emph{Applications of contravariantly finite subcategories}. Adv. Math. 86 (1991), no. 1, 111-152. https://doi.org/10.1016/0001-8708(91)90037-8.

\bibitem[ARS95]{ARS95} M. Auslander, I. Reiten and S. O. Smalo. \emph{Representation theory of Artin algebras}. Cambridge Studies in Advanced Mathematics, vol. 36 (Cambridge University Press, Cambridge, 1995). https://doi.org/10.1017/CBO9780511623608.

\bibitem[AS80]{AS80} M. Auslander and S. O. Smalo. \emph{Preprojective modules over artin algebras}. J. Algebra 66 (1980), 61-122. https://doi.org/10.1016/0021-8693(80)90113-1.

\bibitem[AS81]{AS81} M. Auslander and S. O. Smalo. \emph{Almost split sequences in subcategories}. J. Algebra 69 (1981), 426-454. https://doi.org/10.1016/0021-8693(81)90214-3.

\bibitem[Ba04]{Ba04} S. Bazzoni, \emph{A characterization of n-cotilting and n-tilting modules}, J. Algebra 273 (2004), 359-372. https://doi.org/10.1016/S0021-8693(03)00432-0.

\bibitem[BMPS19]{BMPS19} V. Becerril, O. Mendoza, M. A. P\'erez and V. Santiago. \emph{Frobenius pairs in abelian categories}. Journal of Homotopy and Related Structures (2019) 14:1-50. https://doi.org/10.1007/s40062-018-0208-4.

\bibitem[Ha88]{Ha88} D. Happel. \emph{Triangulated categories in the Representation Theory of Finite Dimensional Algebras}. LMS, 119, Cambridge University Press, 208 pages, (1988). https://doi.org/10.1017/CBO9780511629228.

\bibitem[Iy07]{Iy07} O. Iyama. \emph{Higher-dimensional Auslander-Reiten theory on maximal orthogonal subcategories}. Adv. Math. 210 (2007), no. 1, 22-50. https://doi.org/10.1016/j.aim.2006.06.002.

\bibitem[Ja15]{Ja15} G. Jasso. \emph{Reduction of $\tau$-tilting modules and torsion pairs}. Int. Math. Res. Not. IMRN 16 (2015), 7190-7237. https://doi.org/10.1093/imrn/rnu163.

\bibitem[JJ20]{JJ20} K. M. Jacobsen, P. Jorgesen. Maximal $\tau_d$-rigid pairs. Journal of Algebra 546 (2020) 119-134.

\bibitem[MS06]{MS06} O. Mendoza, E.C. S\'aenz. Tilting categories with applications to stratifying systems. Journal of Algebra, 302 (2006), 419-449.

\bibitem[MSSS13]{MSSS13} O. Mendoza, E.C. S\'aenz, V. Santiago, M. J. Souto. \emph{Auslander-Buchweitz approximation theory for triangulated categories}. Appl. Catego. Struct. (2013) 21:119-139.

\bibitem[Mi86]{Mi86} Y. Miyashita. \emph{Tilting modules of finite projective dimension}. Math. Z. 193 (1986), 113-146. http://eudml.org/doc/173773.

\bibitem[RS89]{RS89} J. Rickard and A. Schofield. \emph{Cocovers and tilting modules}. Math. Proc. Cambr. Phil. Soc. 106 (1989), 1-5. https://doi.org/10.1017/S0305004100067931.

\bibitem[Su21]{Su21} P. Suarez. \emph{On the global dimension of the endomorphism algebra of a $\tau$-tilting module}. Journal of Pure and Applied Algebra, 2021, vol. 225, no 11, p. 106740. https://doi.org/10.1016/j.jpaa.2021.106740.

\bibitem[Ve96]{Verdier} J. L. Verdier. \emph{Des categories Derivees, Des categories Abeliennes}. Societe Mathematique de France, Asterisque, 239, 256 pages, (1996).

\bibitem[We13]{We13} J. Wei. \emph{Semi-tilting complexes}. Israel Journal Math. 194 (2013), 871-893. https://doi.org/10.1007/s11856-012-0093-1.

\bibitem[PX09]{PX09} S. Pan, C. Xi. Finiteness of finitistic dimension is invariant under derived equivalences. Journal of Algebra, 322 (2009), 21-24.

\bibitem[Zh22]{Zh22} X. Zhang. \emph{Self-orthogonal $\tau$-tilting modules and tilting modules}. Journal of Pure and Applied Algebra, Volume 226, Issue 3, 2022. https://doi.org/10.1016/j.jpaa.2021.106860.

\bibitem[Zi14]{Zi14} A. Zimmermann. \emph{Representation Theory a Homological Algebra Point of View}. Springer-Verlag Berlin Heidelberg, 2014. https://doi.org/10.1007/978-3-319-07968-4.\end{thebibliography}
\end{document}